\documentclass[10pt,journal,letterpaper]{IEEEtran}
\usepackage{array}
\newcolumntype{M}[1]{>{\centering\arraybackslash}m{#1}}
\newcolumntype{N}{@{}m{0pt}@{}}

\usepackage{subfigure}
\usepackage{amssymb, amsmath, amsfonts}
\usepackage{empheq}
\usepackage{stfloats}
\usepackage{caption, cite}
\usepackage{graphicx}

\usepackage{algorithm}
\usepackage{algorithmic}

\usepackage{amsthm}
\usepackage{tikz}
\usepackage{tikzscale}
\usepackage{amsmath}
\allowdisplaybreaks
\usetikzlibrary{shadows,arrows,positioning}
\pgfdeclarelayer{background}
\pgfdeclarelayer{foreground}
\pgfsetlayers{background,main,foreground}

\DeclareMathOperator{\diag}{diag}
\DeclareMathOperator{\range}{Range}

\usetikzlibrary{external}
\usepackage{pgfplots}
\usepackage{cases}

\newtheorem{theorem}{Theorem}
\newtheorem{definition}{Definition}
\newtheorem{lemma}{Lemma}

\newtheorem{remark}{Remark}
\newtheorem{assumption}{Assumption}

\newlength\figureheight
\newlength\figurewidth

\DeclareFontFamily{OT1}{pzc}{}
\DeclareFontShape{OT1}{pzc}{m}{it}{<-> s * [1.000] pzcmi7t}{}
\DeclareMathAlphabet{\mathpzc}{OT1}{pzc}{m}{it}

\newcommand{\R}{{\mathbb{R}}}
\newcommand{\E}{{\mathbb{E}}}

\newcommand{\vv}{{\mathbf{v}}}

\newcommand{\bv}{{\mathbf{\bar{v}}}}
\newcommand{\x}{{\mathbf{x}}}
\newcommand{\g}{{\mathbf{g}}}
\newcommand{\bg}{{\mathbf{\bar{g}}}}
\newcommand{\bx}{{\mathbf{\bar{x}}}}
\newcommand{\hx}{{\mathbf{\hat{x}}}}

\newcommand{\lf}{{\nabla{f}}}

\newcommand{\lb}{{\bar{\lambda}}}

\newcommand{\bl}{{\mathbf{L}}}
\newcommand{\bi}{{\mathbf{I}}}
\newcommand{\bk}{{\mathbf{K}}}
\newcommand{\bp}{{\mathbf{P}}}
\newcommand{\bh}{{\mathbf{H}}}

\newcommand\addtag{\refstepcounter{equation}\tag{\theequation}}

\makeatletter

\newcommand{\Rmnum}[1]{\expandafter\@slowromancap\romannumeral #1@}
\makeatother

\title{ \hspace*{\fill} \\\hspace*{\fill} \\ \LARGE \bf{Compressed Distributed Stochastic Nonconvex Optimization with Differential Privacy}}

 \author{Antai Xie, Xiaoqiang Ren, Xinlei Yi, Tao Yang, and Xiaofan Wang
 \thanks{A. Xie, X. Ren, and X. Wang are with the School of Mechatronic Engineering and Automation, Shanghai University, Shanghai, China. Emails: \{xatai,\,xqren,\,xfwang\}@shu.edu.cn.}
 \thanks{X. Yi is with the Shanghai Institute of Intelligent Science and Technology, Tongji University, Shanghai, 201210, China. Email: xinleiyi@tongji.edu.cn.}
 \thanks{T. Yang is with the State Key Laboratory of Synthetical Automation for Process Industries, Northeastern University, Shenyang 110819, China. Email: yangtao@mail.neu.edu.cn.}}

\begin{document}
	\maketitle
	 \begin{abstract}
 This paper studies distributed stochastic nonconvex optimization problems with compressed communication and differential privacy, in which  each agent aims to minimize the sum of all agents' cost functions by using local compressed information exchange. To this end, we propose a compressed distributed stochastic gradient descent algorithm, which is robust under a 
general class of compression operators that allow both relative and absolute compression errors. We then show that the proposed algorithm finds the first-order stationary point for smooth nonconvex functions with the linear speedup convergence rate $\mathcal{O}(1/\sqrt{nT})$ and converges to the optimum if the global cost function additionally satisfies the Polyak--Łojasiewicz (P--\L) condition with the convergence rate $\mathcal{O}(1/(nT^\theta)),\theta\in(0,1)$, where $T$ is the total number of iterations and $n$ is the number of agents. Furthermore, if the P--\L ~constant is known in advance, we show that the proposed algorithm achieves a convergence rate $\mathcal{O}(1/(nT))$. Finally, we show that the proposed algorithm is able to achieve $(0,\delta)$-differential privacy without sacrificing convergence accuracy. Numerical experiments are carried out to verify the efficiency of our algorithm.
	\end{abstract}
	
\section{Introduction}
As a foundational framework for networked systems, distributed optimization has become a prominent research topic~\cite{yang2019survey}, playing a fundamental role in fields such as distributed resource allocation~\cite{xu2017distributed}, control~\cite{nedic2018distributed}, learning~\cite{li2020distributed}, and estimation~\cite{cattivelli2009diffusion}. As noted by Notarstefano \textit{et al.}~\cite{notarstefano2019distributed}, distributed optimization designs local computation and communication rules for the networked processes of multi-agent systems, enabling agents to collaboratively address global problems. To address the distributed optimization problem, researchers have proposed numerous distributed optimization algorithms. A pioneering work is the Distributed Gradient Descent (DGD) algorithm proposed by Nedić and Ozdaglar \cite{nedic2009distributed}. They proved that the DGD algorithm can asymptotically converge to the global optimal solution under a diminishing step size. However, the use of a diminishing step size results in a relatively slow convergence speed for the DGD algorithm. To improve convergence speed, Shi et al. \cite{shi2015extra} proposed a novel EXTRA algorithm by utilizing historical information and demonstrated that EXTRA can linearly converge to the optimal solution. Furthermore, Qu and Li \cite{qu2017harnessing} introduced a distributed optimization algorithm based on gradient tracking by incorporating additional communication to track the global gradient, which also achieves linear convergence. 

It is noteworthy that most of the aforementioned methods require full gradient information. However, such information is often unavailable or difficult to obtain~\cite{conn2009introduction,agarwal2010optimal}. An effective solution is to use stochastic gradients as a substitute for actual gradients, as stochastic gradients can be computed from randomly sampled data subsets. Incorporating stochastic gradients, researchers have derived several convergence results for algorithms under strongly convex conditions~\cite{pu2018swarming,pu2021distributed,xin2019distributed}. For instance, Pu and Garcia~\cite{pu2018swarming} investigated distributed asynchronous stochastic optimization algorithms. Notably, due to the errors introduced by stochastic gradients, these methods can only achieve linear convergence to a neighborhood of the optimal solution. To mitigate the impact of stochastic gradients on convergence, a classical approach is to adopt a time-decaying step size~\cite{lei2018asymptotic,sundhar2010distributed,r156,srivastava2011distributed}. For example, Lei \textit{et al.} \cite{lei2018asymptotic} addressed distributed stochastic convex optimization in random networks and proposed a distributed stochastic convex optimization algorithm with a decaying step size. They proved that the proposed algorithm can almost surely converge to the optimal solution. Furthermore, Lian \textit{et al.} \cite{lian2017can} proposed an algorithm for nonconvex objective functions that achieves a convergence rate of $\mathcal{O}(1/\sqrt{nT})$, where $T$ is the number of iterations. Notably, this convergence rate is $n$ times faster than the optimal convergence rate of centralized Stochastic gradient descent (SGD) algorithms, $\mathcal{O}(1/\sqrt{T})$ \cite{ghadimi2013stochastic}. Researchers have defined this property as the algorithm achieving linear speedup with respect to the number of agents. Similar results on linear speedup under nonconvex objective functions can also be found in the literature \cite{tang2018d,yu2019linear,r139,tang2018communication}. For convex objective functions, Koloskova et al. \cite{koloskova2019decentralized1} established a similar linear speedup with a convergence rate of $\mathcal{O}(1/(nT))$, which is also $n$ times faster than the optimal convergence rate of centralized SGD algorithms, $\mathcal{O}(1/T)$ \cite{rakhlin2011making}.

 In distributed optimization problems, each agent needs to exchange information with its neighbors in order to obtain the global information. However, network bandwidth is typically limited in practical problems. Therefore, it is necessary to consider communication-efficient algorithms. A common solution for agents is to transmit compressed information instead of the raw information. Alistarh \textit{et al.}~\cite{alistarh2017qsgd} and Koloskova \textit{et al.}~\cite{koloskova2019decentralized} proposed communication-efficient SGD algorithms by using an unbiased compressor and biased but contractive compressors, respectively. Singh~\textit{et al.}~\cite{singh2022sparq} additionally considered an event-triggered mechanism to further reduce communication costs. Furthermore, \cite{alistarh2017qsgd,koloskova2019decentralized,singh2022sparq,vogels2020practical} also achieved an $\mathcal{O}(1/\sqrt{nT})$ convergence rate, where the omitted parameters are not affected by the number of agents $n$. Therefore, they achieved linear speedup convergence. However, the authors of~\cite{alistarh2017qsgd,koloskova2019decentralized,singh2022sparq,vogels2020practical} provided analysis only for strongly convex and smooth nonconvex cost functions, but did not provide analysis for the Polyak–Łojasiewicz (P--Ł) condition. The P--Ł condition is weaker than the strong convexity and does not imply the convexity~\cite{yi2021linear}. 
 
 In addition to the aforementioned demand for improving communication efficiency in distributed optimization, how to preserve the privacy of agents has also received widespread attention. Differential privacy, introduced by Dwork et al.~\cite{dwork2008differential}, has emerged as the gold standard for privacy in distributed optimization owing to its strong and mathematically rigorous guarantees. A common approach to achieving differential privacy is to inject Gaussian or Laplacian noise into the information exchanged between agents~\cite{huang2024differential,huang2015differentially,wang2023efficient,ding2021differentially1}. Moreover, several recent works have achieved both communication efficiency and differential privacy by combining compression with additive noise perturbation~\cite{xie2023differentially,xie2024communication,chen2024differentially}. However, these methods typically rely on injecting additional noise to achieve privacy guarantees. To this end, several recent works have achieved both communication efficiency and noise-free differential privacy by employing improved compressors~\cite{wang2022quantization,huo2024compression,agarwal2018cpsgd}. Although these methods cleverly exploit properties of the compressors to achieve differential privacy, they are limited to some specific compressors.  
	 
To relax this restriction on the choice of compressors, in this paper we propose a Robust Compressed Primal--dual SGD algorithm (RCP-SGD) to solve the distributed stochastic nonconvex optimization problem with limited bandwidths and differential privacy. The main contributions of this work are summarized as follows.
 \begin{enumerate}
	\item The proposed algorithm RCP-SGD  is robust for a general class of compressors with both relative and absolute compression errors, which covers the class of compressors used in~\cite{alistarh2017qsgd,koloskova2019decentralized,singh2022sparq,vogels2020practical,koloskova2019decentralized1,xie2023differentially,xie2024communication,chen2024differentially,wang2022quantization,huo2024compression}. We show that RCP-SGD finds a first-order stationary point with the linear speedup convergence rate $\mathcal{O}(1/\sqrt{nT})$ when the cost functions are smooth (Theorem~\ref{theo:convergence1}). We would like to highlight that, comparing with~\cite{alistarh2017qsgd,koloskova2019decentralized,singh2022sparq,vogels2020practical,koloskova2019decentralized1}, we achieve such linear speedup convergence under weaker assumptions on the (stochastic) gradients (Remark~\ref{remark:compare}).
        \item We further prove that if the global objective function additionally satisfies the P--Ł condition, the RCP-SGD algorithm converges to a global optimum at a rate of $\mathcal{O}(1/(nT^\theta))$ ~(Theorem~\ref{theo:convergence2}), where $\theta \in (0,1)$. Moreover, if the P--Ł condition constant is known, it can be proven that the proposed algorithm finds a global optimum at the inear speedup convergence rate of $\mathcal{O}(1/(nT))$ (Theorem~\ref{theo:convergence3}). 
        \item Finally, realize differential privacy of the algorithm through compression, we propose a transformation of the compressor (see Definition~\ref{def:privacycom}). For any compressor satisfying the general compressor assumptions in this paper, it is proven that the RCP-SGD algorithm, under the transformed compressor, can simultaneously achieve $(0,\delta)$-differential privacy and accurate convergence (Theorem~\ref{theo:privacy1}).  Notably, unlike existing results~\cite{wang2022quantization,huo2024compression}, the proposed algorithm is applicable to a general class of compressors.  
 \end{enumerate}

The remainder of this paper is organized as follows. In Section~\ref{sec:Problemsetup}, we introduce the necessary notations and formulate the considered problem. The RCP-SGD algorithm is proposed in Section~\ref{sec:Algorithm1}, and its convergence rate without and with P--Ł condition are then analyzed. Section~\ref{sec:privacy} provides the privacy analysis for RCP-SGD under the transformed compressor. Some numerical examples are provided in Section~\ref{simulation} to verify the theoretical results. The conclusion and proofs are provided in Section~\ref{conclusion} and Appendix, respectively.
	
\emph{Notations}: $\R$ ($\mathbb{R}_{+}$) is the set of (positive) real numbers. $\mathbb{N}$ ($\mathbb{N}^+$) the set of nonnegative (positive) nature numbers. $\mathbb{R}^n$ is the set of $n$ dimensional vectors with real values. The transpose of a matrix $P$ is denoted by $P^\top$, and we use $P_{ij}$ to denote the element in its $i$-th row and $j$-th column. The Kronecker production is denoted by~$\otimes$. The $n$-dimensional all-one and all-zero column vectors are denoted by $\mathbf{1}_n$ and $\mathbf{0}_n$, respectively. The $n$-dimensional identity matrix is denoted by $I_n$. $diag(x)$ is a diagonal matrix with the vector $x$ on its diagonal. We then introduce two stacked vectors: for a vector $\x\in\R^{nd}$, we denote $\bar{x}=\frac{1}{n}(\mathbf{1}_n^\top\otimes I_d)\x$ and $\mathbf{\bar{\x}}\triangleq\mathbf{1}_n\otimes\bar{x}$. $\vert\cdot\vert$ and $\Vert\cdot\Vert$ denote the absolute value and $l_2$ norm, respectively. For a matrix $W$, we use $\bar{\lambda}_W$ and $\underline{\lambda}_W$ to denote its spectral radius and
minimum positive eigenvalue if the matrix $W$ has positive eigenvalues, respectively. Furthermore, for any square matrix $A$ and vector $x$ with suitable dimension, we denote $\Vert x\Vert_A^2=x^\top Ax$.

\section{Preliminaries and Problem Formulation} \label{sec:Problemsetup}

\subsection{Distributed Stochastic Optimization}

In this paper, we consider a network of $n$ agents, and all agents aim to solve the following empirical risk minimization problem~\cite{zhao2022beer}:
\begin{align}\label{P2}
	\min_{x\in\mathbb{R}^d}f(x)=\frac{1}{n}\sum_{i=1}^n f_i(x), f_i(x)=\E_{\xi_{i}\sim\mathcal{D}_i}[F_i(x,\xi_i)],
\end{align}
where $x$ is the global decision variable, $f_i(x):\R^d\mapsto \R$ is the local cost function of agent $i$, $\xi_i$ represents the local data of agent $i$ that follows the distribution $\mathcal{D}_i$ and $F_i(x,\xi_i)$ is a local stochastic cost function. In this paper, we assume each agent $i$ maintains a local estimate $x_{i,k}\in\mathbb{R}^d$ of $x$ at time step $k$ and use $\lf_i(x_{i,k})$ to denote the gradient of $f_i$ with respect to $x_{i,k}$. Furthermore, each agent in the network only has access to the stochastic gradient of its local cost function. We use $\tilde{\nabla}f_{i,k}=\nabla F_i(x_{i,k},\xi_{i,k})$ to denote the stochastic gradient at $x_{i,k}$ with a local data $\xi_{i,k}$.

To solve the global stochastic optimization~\eqref{P2}, agents need to communicate for eastimating the global information. We assume that $n$ agents communicate over an undirected graph $\mathcal{G}(\mathcal{V}, \mathcal{E})$, where $\mathcal{V}=\{1, 2, \ldots, n\}$ is the set of agents' indices and $\mathcal{E} \subseteq \mathcal{V} \times \mathcal{V}$ is the set of edges. The edge $(i,j)\in\mathcal{E}$ if and only if agents~$i$ and~$j$ can communicate with each other. The coupling weight matrix of $\mathcal{G}$ is denoted by $W=[w_{ij}]_{n\times n}\in\mathbb{R}^{n\times n}$ with $w_{ij}>0$ if $(i,j)\in\mathcal{E}$, and $w_{ij}=0$, otherwise. Furthermore, the neighbor agent set of agent $i$ is denoted by $\mathcal{N}_i=\{j\in\mathcal{V}|~(i,j)\in\mathcal{E}\}$. The degree matrix is denoted as $D=diag[d_1,d_2,\cdots,d_n]$, where $d_i=\sum_{j}^n w_{ij},~\forall i\in\mathcal{V}$. The Laplacian matrix of graph $\mathcal{G}$ is denoted by $L=D-W$. Then, the following standard assumptions are given.  

\begin{assumption}\label{as:strongconnected}
	The undirected graph $\mathcal{G}(\mathcal{V}, \mathcal{E})$ is connected. Each local cost function $f_i$ is $L_f$-smooth, for some $L_f>0$, namely for any $x,y\in\mathbb{R}^d$,
	\begin{align}\label{eqn:smooth}
		\left\Vert \lf_i(x)-\lf_i(y)\right\Vert\leq L_f\left\Vert x-y\right\Vert.
	\end{align}
\end{assumption}
From~\eqref{eqn:smooth}, we have
\begin{align}\label{eqn:smooth1}
	\vert f_i(y)-f_i(x)-(y-x)^\top\lf_i(x)\vert\leq\frac{L_f}{2}\left\Vert y-x\right\Vert^2.
\end{align}

Assumption~\ref{as:strongconnected} is standard for distributed optimization problems and widely used in existing works, e.g.,~\cite{shi2015extra,yang2019survey,yi2022communication}.  
\begin{assumption}\label{as:boundedvar}
	The random variables $\{\xi_{i,k},i\in\mathcal{V},k\in\mathbb{N}\}$ are independent of each other. The stochastic gradient $\nabla F_i(x,\xi_{i,k})$ is unbiased, that is,
	\begin{align}
		\E_{\xi_{i,k}}[\nabla F_i(x,\xi_{i,k})]=\lf_i(x),~\forall i\in\mathcal{V},~k\in\mathbb{N}, x\in\R^,
	\end{align}
	where $\mathbb{E}_{\xi_{i,k}}$ denotes the expectation with respect to $\xi_{i,k}$. Furthermore, there exists a constant $\sigma>0$ such that 
	\begin{align}
		\E_{\xi_{i,k}}\Vert \nabla F_i(x,\xi_{i,k})-\lf_i(x)\Vert^2\leq\sigma^2,~\forall i\in\mathcal{V},~k\in\mathbb{N}, x\in\R^d.
	\end{align}
\end{assumption}
\begin{remark}
	Assumption \ref{as:boundedvar} are commonly used for stochastic gradients, e.g.,~\cite{alistarh2017qsgd,singh2022sparq,huang2023cedas}. Furthermore, Assumption~\ref{as:boundedvar}~only requires that the random gradient has a bounded variance, which is weaker than the bounded second moment or the bounded gradient used in~\cite{koloskova2019decentralized1,stich2018local}.
\end{remark}

We then make the following assumptions on the global cost function~$f$.

\begin{assumption}\label{as:finite}
	Let $f^*$ be the minimum function value of the problem~\eqref{P2}. We assume $f^*>-\infty$.
\end{assumption}
\begin{assumption}\label{as:PLcondition}
	(Polyak–Łojasiewicz (P--Ł) condition~\cite{yi2022communication}) There exists a constant $\nu>0$ such that for any $x\in\mathbb{R}^d$,
	\begin{align}
		\frac{1}{2}\left\Vert \lf(x)\right\Vert^2\geq \nu(f(x)-f^*).
	\end{align}
\end{assumption}
\begin{remark}
	Note that the P--Ł condition does not imply the convexity of the global cost function $f$, and is weaker than strong convexity~\cite{yi2021linear}. Furthermore, it is easy to check that all stationary points of~\eqref{P2} under P--Ł condition are the global minimizer.
\end{remark}

\subsection{Differential Privacy and Compression Method}
	
In this paper, we follow the standard setting where each agent aims to protect the privacy of its local data. To quantify the level of privacy preservation, we introduce the following concept of adjacency and differential privacy of agent $i$~\cite{chen2024local}.
\begin{definition}\label{def:adjacency}
	(Adjacency) For any agent $i\in\mathcal{V}$, given two local datasets $\mathcal{S}_i^{(1)}=\{\xi_{i,t}^{(1)},t\in[0,\infty) \}_{i=1}^n$ and $\mathcal{S}_i^{(2)}=\{\xi_{i,t}^{(2)},t\in[0,\infty) \}_{i=1}^n$, $\mathcal{S}_i^{(1)}$ is said to be adjacent to $\mathcal{S}_i^{(2)}$ if there exists a time step $k\in\mathbb{N}^+$ such that $\xi_{i,k}^{(1)}\neq \xi_{i,k}^{(2)}$ while $\xi_{i,t}^{(1)}=\xi_{i,t}^{(2)}$ otherwise.
\end{definition}

From the above definitions, two local datasets are said to be adjacent if they differ in exactly one data point and are identical otherwise. We now introduce the definition of differential privacy.
\begin{definition}\label{def:differentialprivacy}
	(Differential privacy) Let $\mathbb{M}(\mathcal{S}_i, x_{-i})$ be be an implementation
	of a decentralized algorithm by agent $i$, which takes agent $i$'s dataset $\mathcal{S}_i$ and all received information $x_{-i}$ as input. Then, given $\epsilon\geq0$ and $1\geq\delta\geq0$, for any two adjacent datasets $\mathcal{S}_i^{(1)}$ and $\mathcal{S}_i^{(2)}$, any observation $\mathcal{H}_i\subseteq\range(\mathbb{M})$, the implementation $\mathbb{M}$ keeps ($\epsilon$, $\delta$)-differential privacy if
	\begin{align}\label{eq:dp}
		\mathbb{P}\{\mathbb{M}(\mathcal{S}_i^{(1)},x_{-i})\in\mathcal{H}_i\}\leq e^\epsilon \mathbb{P}\{\mathbb{M}(\mathcal{S}_i^{(2)},x_{-i})\in\mathcal{H}_i\}+\delta,
	\end{align}
	where $\range(\mathbb{M})$ denotes the output domain of $\mathbb{M}$.
\end{definition}
Definition~\ref{def:differentialprivacy} implies that if $\mathbb{M}$ is $(\epsilon,\delta)$-differentially private, then for any agent $i$, the output distributions of $\mathbb{M}$ under any pair of adjacent local datasets $\mathcal{S}_i^{(1)}$ and $\mathcal{S}_i^{(2)}$ are close. In other words, an adversary cannot reliably detect differences in any agent’s local data simply by observing the algorithm’s output. Additionally, ($\epsilon$, $\delta$)-differential privacy can be simplified to ($0$, $\delta$)-differential privacy if $\epsilon=0$. Traditionally, privacy is ensured by adding explicit noise (Laplacian or Gaussian)~\cite{huang2015differentially,ding2021differentially}. However, recent works~\cite{huo2024compression,wang2022quantization} show that some specific compressors can simultaneously save communication resource and preserve differential privacy. Accordingly, we assume agents exchange only compressed variables, achieving both goals without additional noise. More specifically, for any $x\in\R^d$, we consider a general class of stochastic compressors $\mathcal{C}(x)$ that satisfy the following assumption.

\begin{assumption}\label{as:compressor}
For some constants $\varphi\in(0,1]$, $r>0$ and $\sigma_{\mathcal{C}}\geq0$, the compressor $\mathcal{C}(\cdot):\mathbb{R}^d\mapsto\mathbb{R}^d$ satisfies 
\begin{align*}
    \mathbb{E}_\mathcal{C}\left[
        \left\Vert \frac{\mathcal{C}(x)}{r}-x\right\Vert^2\right]\leq(1-\varphi)\left\Vert x\right\Vert^2+\sigma_{\mathcal{C}}, \forall x\in\mathbb{R}^d,\addtag\label{eq:propertyofcompressors}
\end{align*}
where $\mathbb{E}_\mathcal{C}$ denotes the expectation with respect to the stochastic compression operator $\mathcal{C}$. 
\end{assumption}
From~\eqref{eq:propertyofcompressors} and the Cauchy--Schwarz inequality, one obtains that
\begin{align}\label{eq:propertyofcompressors1}
    \mathbb{E}_C\left[
    \left\Vert C(x)-x\right\Vert^2\right]\leq r_0\left\Vert x\right\Vert^2+2r^2\sigma_{\mathcal{C}}, \forall x\in\mathbb{R}^d.
\end{align}
where $r_0=2r^2(1-\varphi)+2(1-r)^2$.

\begin{remark}
Assumption~\ref{as:compressor} follows the formulation in~\cite{liao2024robust}. As highlighted therein, this assumption is less restrictive than the conditions typically required by most existing compressed decentralized optimization algorithms, providing a more general framework for convergence analysis. It covers the deterministic quantization~\cite{zhu2016quantized} and unbiased random quantization~\cite{reisizadeh2019robust}. It is worth noting that compressors under Assumption~\ref{as:compressor} possess both bounded relative and bounded absolute errors simultaneously. The implemented algorithm must be carefully designed to handle the influence of both relative and absolute error errors on convergence. The precise relationship between compressors and privacy are elaborated in Section~\ref{sec:privacy}.
\end{remark}

\section{Compressed Primal--Dual SGD Algorithm }\label{sec:Algorithm1}
In this section, we propose a Robust Compressed Primal--dual SGD algorithm (RCP-SGD) to solve the problem~\eqref{P2}, which is robust on various compressors that satisfy the Assumption~\ref{as:compressor}. Furthermore, we analyze the convergence rates of RCP-SGD without and with the P--Ł condition.
\subsection{Algorithm Description}
To solve the distributed nonconvex optimization problem~\eqref{P2}, Yi~\textit{et~al.}~\cite{yi2022primal} proposed the following distributed primal--dual SGD algorithm
\begin{align*}
&~x_{i,k+1}=x_{i,k}-\eta_k(\gamma_k\sum_{j=1}^n L_{ij}x_{j,k}+\omega_k  v_{i,k}+\tilde{\nabla} f_{i,k}),\addtag\label{iterationxp}\\
&~v_{i,k+1}=v_{i,k}+\eta_k\omega_k  \sum_{j=1}^n L_{ij}x_{j,k},~v_{i,0}=\mathbf{0}_d,\addtag\label{iterationv}
\end{align*}
where $\eta_k$ is step-size, $\gamma_k$ as well as$~\omega_k $ are time-varying positive parameters, and $v_{i,k}$ is the dual variable of agent~$i$. 

To accommodate limited bandwidth, we assume that each agent~$i$ use a estimated compressed state $\hat{x}_{i,k}$ to replace the true information $x_{i,k}$ in updates. Specifically, the updates for agent $i\in\mathcal{V}$ can be described as follows:
\begin{align*}
&~x_{i,k+1}=x_{i,k}-\eta_k(\gamma_k \sum_{j=1}^n L_{ij}\hat{x}_{j,k}+\omega_k  v_{i,k}+\tilde{\nabla} f_{i,k}),\addtag\label{eq:iterationxp2}\\
&~v_{i,k+1}=v_{i,k}+\eta_k\omega_k \sum_{j=1}^n L_{ij}\hat{x}_{j,k},~v_{i,0}=\mathbf{0}_d,\addtag\label{eq:iterationv2}
\end{align*}
where 
\begin{align*}
	&\hat{x}_{j,k}=x_{j,k}^c+h_k\mathcal C((x_{j,k}-x_{j,k}^c)/h_k),\addtag\label{citerationx}\\
	&x_{j,k+1}^c=(1-\alpha_x)x_{j,k}^c+\alpha_x\hat{x}_{j,k},\addtag\label{citerationxc}
\end{align*}
with $\alpha_x$ being a positive parameter, $\{h_k\}$ is a designed sequence and $x_{i,k}^c$,~$\forall i\in\mathcal{V}$ is an auxiliary variable with initial value $x_{i,0}^c=\mathbf{0}_d$. We then describe the RCP-SGD in Algorithm~\ref{Al:RCP-SGD}. From the compressor's property~\eqref{eq:propertyofcompressors1}, the compression error satisfies
\begin{align*}
\E&[\|x_{j,k}-\hat x_{j,k}\|^2]\\
&=\E[h_k^2\|(x_{j,k}-x_{j,k}^c)/h_k-\mathcal{C}((x_{j,k}-x_{j,k}^c)/h_k)\|^2]\\
&\leq r_0\E\|x_{j,k}-x_{j,k}^c\|^2+2r^2\sigma_{\mathcal{C}}h_k^2, \addtag\label{eq:propertyofcompressors0}
\end{align*}

\begin{algorithm}[]
	\caption{RCP-SGD Algorithm}
	\label{Al:RCP-SGD}
	\begin{algorithmic}[1]
		\STATE \textbf{Input:} Stopping time $T$, Laplacian matrix $L$, and positive parameters $\{\eta_k\}$, $\{\gamma_k\}$, $\{\omega_k\}$, $\alpha_x$.
		\STATE \textbf{Initialization:} Each ~$i\in\mathcal{V}$ chooses arbitrarily $x_{i,0}\in\mathbb{R}^d$, $x^c_i(0)=\bf{0}_d$, $v_i(0)=\bf{0}_d$.
		\FOR{$k=0,1,\dots,T-1$}
		\FOR {for $i\in\mathcal{V}$ in parallel} 
		\STATE Compute $C(x_{i,k}-x_{i,k}^c)$ and broadcast it to its neighbors $\mathcal{N}_i$.
		\STATE Receive $C(x_{j,k}-x_{j,k}^c)$ from $j\in\mathcal{N}_i$.	
		\STATE Update $x_{i,k+1}$ and $v_{i,k+1}$ according to~\eqref{eq:iterationxp2} and~\eqref{eq:iterationv2}, respectively.
		\STATE Update $x_{j,k+1}^c$ from~\eqref{citerationxc}.
		\ENDFOR
		\ENDFOR
		\STATE \textbf{Output:} \{$x_{i,k}$\}.
	\end{algorithmic}
\end{algorithm}

\subsection{Convergence Analysis of RCP-SGD}
In this section, we first show the convergence of RCP-SGD  for smooth nonconvex cost functions.

\begin{theorem}\label{theo:convergence1}
	Suppose Assumptions~\ref{as:strongconnected}--\ref{as:finite} and \ref{as:compressor} hold and in Algorithm~\ref{Al:RCP-SGD}, let~$\gamma_k=\beta_1\omega_k,
 ~\eta_k=\frac{\beta_2}{\omega_k},~\omega_k=\omega>\beta_3$, $\alpha_x\in(0,\frac{1}{r})$, and $h_k=h_0^k$, $\forall k\in\mathbb{N}$ where $~\beta_1>c_0,~\beta_2>0$, $h_0\in(0,1)$ is a arbitrary constant with $c_0,\beta_3$ being positive constants given in Appendix~\ref{app-convergence1}. Then, for any $T\in\mathbb{N}$, we have
 \begin{align*}
    &\frac{1}{T}\sum_{k=0}^{T-1}\mathbb{E}\left[\frac{1}{n}\sum_{i=1}^n\Vert x_{i,k}-\bar{x}_k\Vert^2\right]\leq \mathcal{O}(\frac{1}{T})+\mathcal{O}(\frac{1}{\omega^2}),\addtag\label{eq:theo11}\\
    &\frac{1}{T}\sum_{k=0}^{T-1}\mathbb{E}\Vert \nabla f(\bar{x}_k)\Vert^2\leq\mathcal{O}(\frac{\omega}{T})+\mathcal{O}(\frac{1}{n\omega})+\mathcal{O}(\frac{1}{T})+\mathcal{O}(\frac{1}{\omega^2}),\addtag\label{eq:theo12}
 \end{align*}
Let $\omega=\beta_2\sqrt{T}/\sqrt{n}$, for any $T>n(\beta_3/\beta_2)^2$, then we have
  \begin{align*}
    &\frac{1}{T}\sum_{k=0}^{T-1}\mathbb{E}\left[\frac{1}{n}\sum_{i=1}^n\Vert x_{i,k}-\bar{x}_k\Vert^2\right]= \mathcal{O}(\frac{n}{T}),\addtag\label{eq:coro11}\\
    &\frac{1}{T}\sum_{k=0}^{T-1}\mathbb{E}\Vert \nabla f(\bar{x}_k)\Vert^2=\mathcal{O}(\frac{1}{\sqrt{nT}})+\mathcal{O}(\frac{n}{T}).\addtag\label{eq:coro12}
 \end{align*}
\end{theorem}
\begin{proof}
	See Appendix~\ref{app-convergence1}.
\end{proof}
\begin{remark}\label{remark:compare}
    Notably, the omitted parameters in $\mathcal{O}(1/\sqrt{nT})$ in~\eqref{eq:coro12} are unaffected by any parameters related to communication graphs. In other words, RCP-SGD is suitable for any connected graph. According to Theorem~\ref{theo:convergence1}, RCP-SGD achieves the linear speedup convergence rate $\mathcal{O}(1/\sqrt{nT})$ under smooth and nonconvex cost functions. Furthermore, it is important to note that, although the similar linear speedup convergence rate is also established in references~\cite{alistarh2017qsgd,koloskova2019decentralized,singh2022sparq,vogels2020practical,koloskova2019decentralized1}, they require additional assumptions. Specifically, the methods~\cite{alistarh2017qsgd,koloskova2019decentralized,singh2022sparq,koloskova2019decentralized1} required the stochastic gradients have second bounded moment and the method~\cite{ vogels2020practical} assumed that $\|\lf_i(x)-\lf(x)\|^2$ is uniformly bounded. Furthermore, these methods either do not account for compressed communication or consider compressors under conditions that are stricter than Assumption~\ref{as:compressor}.
\end{remark}

Then we provide the linear convergence of RCP-SGD with the P--\L~condition.
\begin{theorem}\label{theo:convergence2}
	Suppose Assumptions~\ref{as:strongconnected}--\ref{as:compressor} hold and in Algorithm~\ref{Al:RCP-SGD}, for any $T\geq(\beta_3/\beta_2)^{1/\theta}\in\mathbb{N}$, let~$\gamma_k=\beta_1\omega_k,
 ~\eta_k=\beta_2/\omega_k,~\omega_k=\omega=\beta_2(T+1)^\theta$, $\alpha_x\in(0,\frac{1}{r})$, and $h_k=h_0^k$, $\forall k\in\mathbb{N}$ where $~\beta_1>c_0,~1>\beta_2>0$, and $\theta\in(0,1)$, $h_0\in(0,1/2)$ are arbitrary constants with $c_0$ and $\beta_3$ being positive constants given in Appendix~\ref{app-convergence1}. Then, we have
 \begin{align*}
 \mathbb{E}\left[\frac{1}{n}\sum_{i=1}^n\Vert x_{i,T}-\bar{x}_T\Vert^2\right]= \mathcal{O}(\frac{1}{T^{2\theta}}),\addtag\label{eq:theo21}\\
 \mathbb{E}\left[f(\bar{x}_T)-f^*\right]= \mathcal{O}(\frac{1}{nT^{\theta}})+ \mathcal{O}(\frac{1}{T^{2\theta}}).\addtag\label{eq:theo22}\\
 \end{align*}
\end{theorem}
\begin{proof}
	See Appendix~\ref{app-convergence2}.
\end{proof} 

From the above theorem, it can be seen that the RCP-SGD algorithm converges to the global optimum under the P--\L~condition. However, its convergence rate is strictly slower than $\mathcal{O}(1/(nT))$ and requires an upper bound on the number of iterations $T$. To overcome these limitations, we presents the following result: when the P--\L~constant is known, our proposed algorithm achieves linear speedup with an $\mathcal{O}(1/(nT))$ convergence rate to the global optimum without imposing any restriction on the iteration count $T$.  

 \begin{theorem}\label{theo:convergence3}
 	Suppose Assumptions~\ref{as:strongconnected}--\ref{as:compressor} hold, and P--Ł constant $\nu$ is known in advance. In Algorithm~\ref{Al:RCP-SGD}, let~$\gamma_k=\beta_1\omega_k,~\omega_k=\beta_0(k+t_1)$, $\alpha_x\in(0,\frac{1}{r})$, $\eta_k=\frac{\beta_2}{\omega_k}$, and $h_k=h_0^k$, $\forall k\in\mathbb{N}$ where $\beta_0\in[\tilde{c}\nu\beta_2/4,\nu\beta_2/4)$, $\beta_1>\bar{c}_1$, $0<\beta_2<\bar{c}_2$, $t_1>\bar{c}_5$, $h_0\in(0,1/t_1)$ with $1>\tilde{c}>0$ being a constant, $\bar{c}_1$, $\bar{c}_2$, and $\bar{c}_5$ being constants given in Appendix~\ref{app-convergence3}. Then, for any $T\in\mathbb{N}$, we have
 \begin{align*}
     &\E[\frac{1}{n}\sum_{i=1}^n\Vert x_{i,T}-\bar{x}_T\Vert^2]=\mathcal{O}(\frac{1}{T^2}),\addtag\label{eq:convergenceofth31}\\
     &\E[f(\bar{x}_{k})-f^*)]=\mathcal{O}(\frac{1}{nT})+\mathcal{O}(\frac{1}{T^2}).\addtag\label{eq:convergenceofth32}
 \end{align*}
\begin{proof}
     See Appendix~\ref{app-convergence3}.
 \end{proof}
\end{theorem}
 \begin{remark}
     From \eqref{eq:convergenceofth32}, it follows that RCP-SGD achieves a linear speedup convergence rate $\mathcal{O}(1/(nT))$. Notably, the same linear speedup properties has been established in~\cite{reisizadeh2020fedpaq,lan2020communication} and~\cite{yi2022primal}. However, the works \cite{reisizadeh2020fedpaq,lan2020communication} assume convexity of the cost function, while \cite{yi2022primal} does not consider compressed communication.
 \end{remark}

 \subsection{Proof Sketch}\label{sec:proofsketch}
 This subsection provides the proof sketch for Theorems~\ref{theo:convergence1}--\ref{theo:convergence3}. First, for notational convenience, we denote $\g_k^s,\x_k, \x^c_k$, and $\bx_k$ as the compact form of $\tilde{\nabla} f_{i,k},x_{i,k},x_{i,k}^c$, and $\bar{x}_k$, respectively, and denote $\g_k^b=\nabla\tilde{f}(\bx_{k})$. To guarantee the convergence of the algorithm, we construct a Lyapunov function $V_k=\sum_{i=1}^5 V_{i,k}$, which consists of consensus errors $\|\x_k-\bx_k\|^2$, compression errors $\Vert\x_{k}-\x^c_{k}\Vert^2$, optimization errors $n(f(\bar{x}_{k})-f^*)$ and the intermediate coupling terms of these errors (the detailed definition of $V_k$ is given in Appendix~\ref{app-convergence0}). We then analyze the relationship between these five terms at time $k+1$ and time $k$, and present these results in Lemma~\ref{lemma:lyapunov}. It is noteworthy that analyzing the Lyapunov function $V_k$ using the methods mentioned above is nontrivial. This is attributed to two key factors: (i) The definition of $V_k$ shows a coupling between the compression errors and the stochastic gradient errors, which undeniably complicates the analysis; (ii) The compressor considered in this paper exhibits both bounded relative and bounded absolute compression errors (see \eqref{eq:propertyofcompressors0}), and RCP-SGD is required to ensure that both types of errors contract. Consequently, due to these reasons, we have to use more stricter inequalities and tighter parameter designs to ensure the convergence.

\section{Adapted compressors enable differential privacy}\label{sec:privacy}
 In this section, we show that compression can ensure differential privacy. From Definition~\ref{def:differentialprivacy}, it is easy to know that differential privacy relies on the uncertainty or obfuscation introduced by the randomized mechanism $\mathbb{M}$. Combining the characteristics of stochastic compressors, it is natural to consider how to use compression to protect privacy. This is also the main problem we aim to investigate in this section.

It is crucial to note that merely relying on the uncertainty inherent in Assumption~\ref{as:compressor} is insufficient to ensure privacy. This is because Assumption~\ref{as:compressor} is general enough to encompass the standard uncompressed case. Specifically, if we consider the boundary case where $\varphi=1$ and $\sigma_{\mathcal{C}}=0$, Assumption \ref{as:compressor} degenerates to the uncompressed setting, i.e., $\mathcal{C}(x)=x$. Clearly, privacy cannot be guaranteed in this scenario. While improved compressors have been proposed to simultaneously achieve both convergence and privacy, such results are contingent on a specific compressor structure. This reliance on a particular compressor is a common limitation faced by existing methods~\cite{wang2022quantization,huo2024compression,agarwal2018cpsgd}. Therefore, to enhance the algorithm's privacy across a wider range of compressors, this section focuses on transforming a general class of compressors satisfying Assumption \ref{as:compressor} to strengthen the privacy guarantees of the RCP-SGD algorithm.

By Definition~\ref{def:differentialprivacy}, to achieve differential privacy, the outputs of the compressor for different inputs must be statistically similar. To this end, we propose the following transformation for compressors. Specifically, any compressor that satisfies Assumption \ref{as:compressor} can be improved through the transformation method defined below, thereby enhancing privacy.
\begin{definition}~\label{def:privacycom}
Given any compressor $\mathcal{C}(x)$ satisfying Assumption \ref{as:compressor}, the privacy-enhanced transformation for $\mathcal{C}'(x)$ is defined as follows.
\begin{align}\label{eq:compressedpev}
    \mathcal{C}'(x)= \begin{cases} \mathcal{C}(x) , & w.p.~~~1-q\\ 0, & w.p.~~~q\end{cases}
\end{align}
where $0<q<1$ is the parameter used to control the probability of confusion.
\end{definition}
\begin{remark}
According to Definition \ref{def:privacycom}, the essence of the transformed compressor $\mathcal{C}'(x)$ is to output $0$, which is independent of the input $x$, with a positive probability $q$. From Definition~\ref{def:privacycom} and Assumption~\ref{as:compressor}, one obtains that
\begin{align*}
    \mathbb{E}_\mathcal{C'}\left[
        \left\Vert \frac{\mathcal{C}'(x)}{r}-x\right\Vert^2\right]&\leq (1-q)(1-\varphi)\left\Vert x\right\Vert^2+(1-q)\sigma_{\mathcal{C}}\\
        &~~~+q\|x\|^2\\
        &\leq (1-\varphi(1-q))\|x\|^2+(1-q)\sigma_{\mathcal{C}}.
\end{align*}
Note that the transformed compressor $\mathcal{C}'(x)$ still satisfies Assumption~\ref{as:compressor}, characterized by $\varphi' = \varphi(1-q)$ and $\sigma_{\mathcal{C}}' = (1-q)\sigma_{\mathcal{C}}$.Consequently, the RCP-SGD algorithm is still compatible with the transformed compressor $\mathcal{C}'(x)$. Furthermore, the transformed compressor $\mathcal{C}'(x)$ can be viewed as a combination of a compressor and an event-triggered algorithm. As per Definition \ref{def:privacycom}, the compressor $\mathcal{C}'(x)$ does not transmit information with probability $q$. Consequently, the transformed compressor $\mathcal{C}'(x)$ also improves communication efficiency compared to the initial compressor $\mathcal{C}(x)$.
\end{remark}

Now, we are ready to show that the privacy can be preserved under such privacy-enhanced compressors.
\begin{theorem}\label{theo:privacy1}
    Suppose Assumptions~\ref{as:strongconnected}--\ref{as:finite} hold. For any compressor $\mathcal{C}(\cdot)$ satisfying Assumption~\ref{as:compressor}, by employing its privacy-enhanced version $\mathcal{C}'(\cdot)$ given in Definition~\ref{def:privacycom} and replacing the parameters $\varphi$ and $\sigma_{\mathcal{C}}$ in Theorem~\ref{theo:convergence1} with $\varphi(1-q)$ and $(1-q)\sigma_{\mathcal{C}}$ respectively, Algorithm~\ref{Al:RCP-SGD} achieves $(0,\delta)$-differential privacy with $\delta=1-q$ for the local data of any agent $i$, while still ensuring convergence.
\end{theorem}
\begin{proof}
    See Appendix~\ref{app-privacy1}.
\end{proof}

\begin{remark}
	According to Definition \ref{def:privacycom}, the approach of the transformed compressor $\mathcal{C}'(x)$ for achieving privacy is similar to existing methods that map sensitive information to other spaces \cite{shoukry2016privacy,lu2018privacy,yan2021distributed,chen2022privacy,zhang2018admm,alexandru2020cloud,lou2017privacy}. However, this method does not require additional computational resources, unlike encryption-based methods \cite{shoukry2016privacy,lu2018privacy,yan2021distributed,chen2022privacy,zhang2018admm,alexandru2020cloud}. Furthermore, compared to the approach proposed by Lou et al. \cite{lou2017privacy}, which protects privacy through projection operations and asynchronous heterogeneous step-size optimization mechanisms, the transformed compressor method here achieves stricter differential privacy (see Theorem \ref{theo:privacy1}) and additionally enhances communication efficiency.
\end{remark}

\begin{remark}
	According to Theorem \ref{theo:privacy1}, for any compressor $\mathcal{C}(x)$ satisfying Assumption \ref{as:compressor}, the RCP-SGD algorithm can simultaneously achieve convergence and guarantee $(0,\delta)$-differential privacy for the local data of any agent $i$ at time step $k$ under its transformed version. Compared to results in \cite{wang2022quantization,huo2024compression}, the privacy parameter $\delta$ in Theorem \ref{theo:privacy1} is independent of the dimension $d$. The privacy analysis here is not limited to specific compressors but is effective for a general class of compressors~(Assumption \ref{as:compressor}). It is worth noting that while the compressor transformation in Definition~\ref{def:privacycom} can simultaneously achieve convergence and differential privacy, a trade-off between privacy level and convergence rate is still required. More specifically, according to Theorem~\ref{theo:privacy1}, a larger parameter $q$ implies a higher level of privacy; however, an increase in $q$ leads to a decrease in the parameter $\varphi$ of the transformed compressor.  By combining Theorem~\ref{theo:convergence1}, \eqref{eq:upperofx3}, and the definitions of parameters $\varphi_1 = \alpha_xr\varphi$ and $\epsilon_9 = 11 + \frac{16}{\varphi_1}$ (detailed in Appendices~\ref{app-convergence0} and~\ref{app-convergence1}), it can be concluded that the convergence rate of the algorithm diminishes as $q$ increases. This illustrates the necessity of a trade-off between the privacy level and the convergence rate.
\end{remark}

\section{simulation}\label{simulation}
This section validates the effectiveness of RCP-SGD through simulation experiments. First, consider a distributed optimization problem with $n=10$ agents communicating over a ring graph. Specifically, all agent aim to address the following nonconvex distributed binary classification problem \cite{yi2021linear,sun2019distributed}.
\begin{align*}
	&~~~~~\min_{x}f(x)=\frac{1}{10}\sum_{i=1}^{10} f_i(x),\\\addtag\label{eq:simulation0}
 &f_i\left(x_i\right)=\frac{1}{m} \sum_{j=1}^m \log \left(1+\exp \left(-u_{i j} x_i^{\top} z_{ij}\right)\right)+\sum_{s=1}^d \frac{\lambda \alpha x_{i, s}^2}{1+\alpha x_{i, s}^2},
\end{align*}
where $z_{ij}\in\R^d$ is the feature vector, $u_{ij}\in\{-1,1\}$ is the label and $x_{i, s}$ is the $s$-th coordinate of $x_i$. We use the breast cancer dataset\footnote{https://archive.ics.uci.edu/dataset/14/breast+cancer. Prior to the experiments, we apply min-max normalization to the data. This preprocessing step benefits the experiments but may alter the variance of the data. Further details can be found in~\cite[Chapter 3]{han2022data}.} to train the model. The goal of the agents is to determine whether a breast cancer is malignant or benign by training the model \eqref{eq:simulation0}. Furthermore, to enhance communication efficiency, this section considers the following compressors.

\begin{table*}[h]
	\centering
	\caption{Parameter setting for different algorithms}
	\begin{tabular}{lc c c c c c c}\hline
		Algorithms & Compressors &$\gamma$&$\omega$&$\eta$&$\alpha_x$&$m_k$&$q$ \\\hline
		DSGD &---&---&---&0.1&---&---&---\\
		Choco-SGD &$\mathcal{C}_1$&0.2&---&0.1&---&---&---\\
		RCP-SGD-1&$\mathcal{C}_1$&5&0.5&$0.08/k^{10^{-2}}$&---&$1/k$&---\\
		RCP-SGD-2&$\mathcal{C}_3$&2&0.5&$0.08/k^{10^{-2}}$&0.8&$1/k$&---\\ RCP-SGD-3 &$\mathcal{C}_2$&2&0.5&$0.08/k^{10^{-2}}$&0.8&$1/k$&---\\ RCP-SGD-4 &$\mathcal{C}_4$&2&0.5&$0.08/k^{10^{-2}}$&0.8&$1/k$&---\\ RCP-SGD-5 &$\mathcal{C}'_3$&2&0.5&$0.08/k^{10^{-2}}$&0.8&$1/k$&0.2\\unRCP-SGD &---&2&0.5&$0.08/k^{10^{-2}}$&0.8&$1/k$&---\\\hline\\
	\end{tabular}
	\label{tab:parameter}
\end{table*}

\begin{itemize}
    \item Biased $b$-bits quantizer\cite{koloskova2019decentralized}:
    \begin{align*}
    	\mathcal C_1(x):=\frac{\Vert x\Vert}{\xi}\cdot \text{sign}(x)\cdot 2^{-(b-1)}\circ \left\lfloor\frac{2^{(b-1)}\vert x\vert}{\Vert x\Vert} +u\right\rfloor,\addtag\label{eq:compressor1}
    \end{align*}
where $\xi=1+\min\{\frac{d}{2^{2(b-1)}},\frac{\sqrt{d}}{2^{(b-1)}}\}$, the vector $u$ is a random dithering vector uniformly sampled from $[0,1]^d$, $\circ$ denotes the Hadamard product, and $\text{sign}(\cdot)$, $|\cdot|$, and $\lfloor\cdot\rfloor$ represent the element-wise sign, absolute value, and floor functions, respectively. In this simulation experiment, the parameter is set to $b=2$.
     \item Sign norm compressor\cite{yi2022communication}:
    \begin{align*}
    	\mathcal C_2(x):=\frac{\Vert x\Vert_\infty}{2}\text{sign} (x).\addtag\label{eq:compressor2}
    \end{align*}
     \item The improved $b$-bits quantizer:
    \begin{align*}
    	\mathcal C_3(x):=\frac{\Phi(\Vert x\Vert)}{\xi}\cdot \text{sign}(x)\cdot 2^{-(b-1)}\circ \left\lfloor\frac{2^{(b-1)}\vert x\vert}{\Vert x\Vert} +u\right\rfloor,\addtag\label{eq:compressor3}
    \end{align*}
Here, $\Phi(x):\mathbb{R}\mapsto\mathbb{Z}$ is a mapping that outputs $\lfloor x\rfloor+1$ with probability $x-\lfloor x\rfloor$ and $\lfloor x\rfloor$ with probability $\lfloor x\rfloor-x+1$.
 \item The improved Sign-norm compressor:
    \begin{align*}
    	\mathcal C_4(x):=\frac{\Phi(\Vert x\Vert_\infty)}{2}\text{sign} (x).\addtag\label{eq:compressor4}
    \end{align*}
\end{itemize}	
It is evident that the aforementioned compressors all satisfy Assumption \ref{as:compressor}. Furthermore, the compressors $\mathcal{C}_3(x)$ and $\mathcal{C}_4(x)$ are improved versions of the compressors $\mathcal{C}_1(x)$ and $\mathcal{C}_2(x)$, respectively. Compared to the original versions, the mapping $\Phi$ compresses real numbers into integers, resulting in fewer bits required for transmission by the compressors $\mathcal{C}_3(x)$ and $\mathcal{C}_4(x)$ compared to $\mathcal{C}_1(x)$ and $\mathcal{C}_2(x)$. However, it is noteworthy that these two improved compressors exhibit both relative and absolute compression errors simultaneously.

Subsequently, the RCP-SGD algorithm is compared with the Distributed SGD algorithm (DSGD) and the compressed algorithm \cite{koloskova2019decentralized1} (Choco-SGD) under different parameters, with specific parameter settings detailed in Table \ref{tab:parameter}. Similar to Chapter 5, the residual $R_k\triangleq\min_{t\leq k}\{\frac{1}{n}\sum_{i=1}^n\Vert x_{i,t}-\bar{x}_t\Vert^2+\|\nabla f(\bar{x}_t) \|^2\}$ is used here to evaluate the convergence of the algorithms. As shown in Figure \ref{fig:bit1}, the RCP-SGD algorithm ensures convergence under different compressors, even for those with both relative and absolute compression errors. Furthermore, it can be observed that the RCP-SGD algorithm, when using the improved compressors $\mathcal{C}_3(x)$ and $\mathcal{C}_4(x)$, requires fewer transmitted bits to achieve the same level of accuracy compared to using the original compressors $\mathcal{C}_1(x)$ and $\mathcal{C}_2(x)$. This indicates that, despite the additional absolute compression error introduced by the mapping $\Phi$, it further enhances communication efficiency. Moreover, compared to the Choco-SGD algorithm, the RCP-SGD algorithm achieves similar or better convergence rates across different compressors and is applicable to a broader range of compressors.

\begin{figure}
	\centering
	\includegraphics[width=1\linewidth]{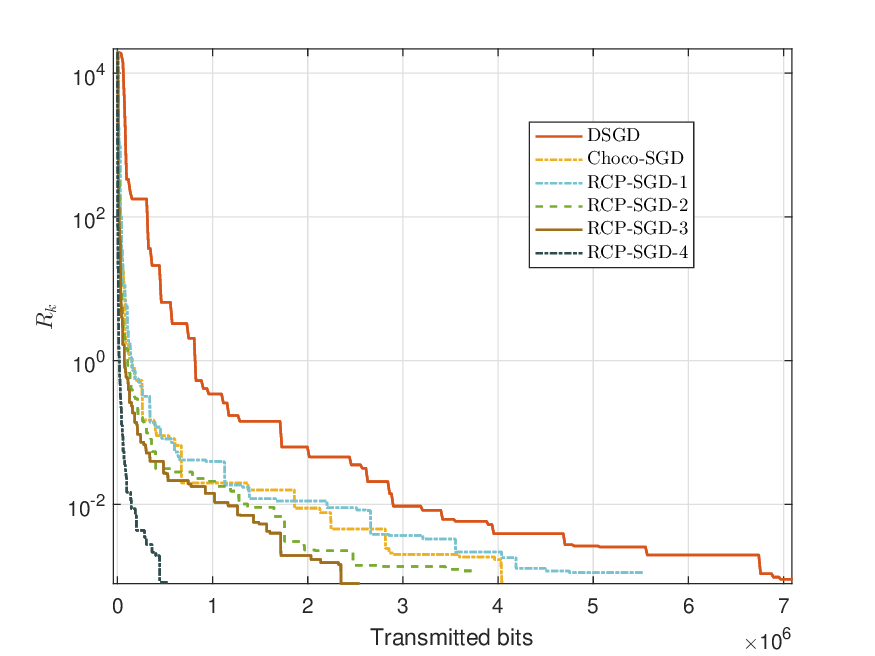}
	\caption{The evolution of residual with respect to the transmitted bits under DSGD, Choco-SGD, and RCP-SGD}
	\label{fig:bit1}
\end{figure}

\begin{figure}
	\centering
    \includegraphics[width=1\linewidth]{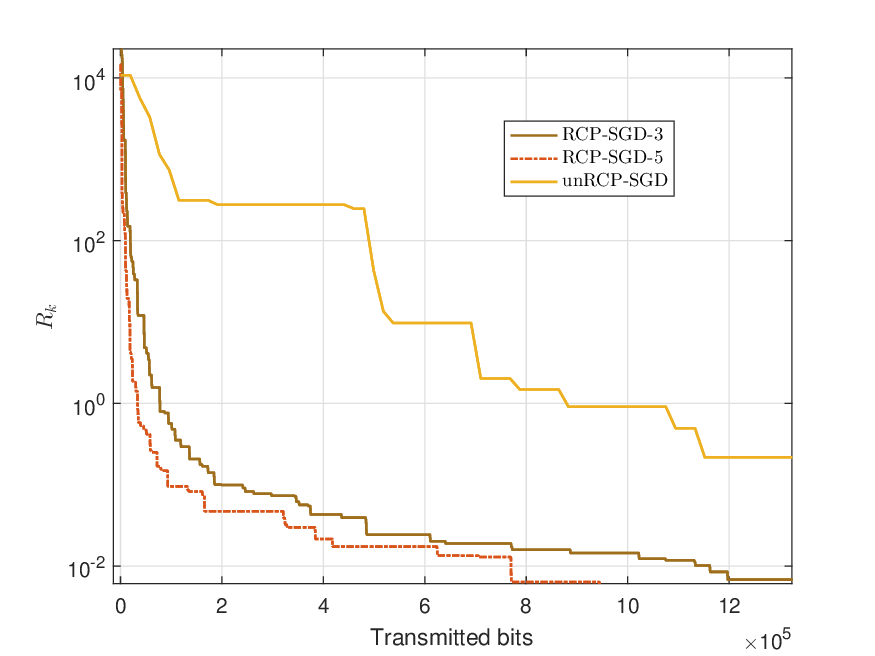}
	\caption{The evolution of residual with respect to the transmitted bits under RCP-SGD-3, RCP-SGD-5, and unRCP-SGD}
	\label{fig:uncompressed}
\end{figure}

Then we further validates the effectiveness of the compressor and its transformed version on the communication efficiency of the algorithm. Specifically, taking the RCP-SGD-3 algorithm as an example, it is compared with the algorithm RCP-SGD-5 under the transformed compressor and the uncompressed RCP-SGD (unRCP-SGD). Similarly, the specific parameter settings can be found in Table \ref{tab:parameter}. As shown in Figure \ref{fig:uncompressed}, both the initial compressor $\mathcal{C}_3$ and its transformed version $\mathcal{C}'_3$ effectively reduce the algorithm's requirement for transmitted bits. Furthermore, it can be observed that the transformed compressor $\mathcal{C}'_3$ outperforms the initial compressor $\mathcal{C}_3$ in terms of saving transmitted bits. This is because the transformed compressor has a probability $q$ of not transmitting information, which can be viewed as a combination of a compressor and an event-triggered mechanism. Consequently, it requires fewer communication resources than the initial compressor. Next, the impact of the topology graph on the RCP-SGD algorithm is validated. Specifically, the convergence of the RCP-SGD-2 algorithm is compared under different topologies: ring, torus, and fully connected graphs. As illustrated in Figure \ref{fig:graph1}, the RCP-SGD-2 algorithm ensures convergence across these different topologies, with the influence of the topology on the algorithm's convergence being minimal.

\begin{figure}
	\centering
	\includegraphics[width=1\linewidth]{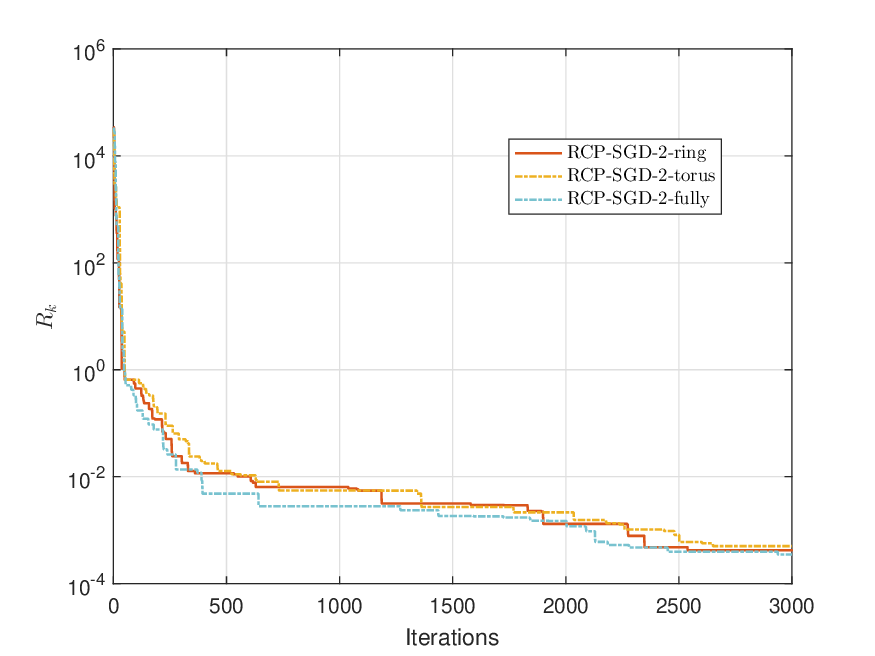}
	\caption{The evolution of residual under RCP-SGD with different graphs}
	\label{fig:graph1}
\end{figure}

To validate the effectiveness of the transformed compressor in the RCP-SGD-5 algorithm with respect to privacy, this section introduces a powerful DLG attack \cite{zhu2019deep}, which can infer the feature vector of problem \eqref{eq:simulation0} through shared gradient information. To evaluate privacy, the attacker's estimation error is defined as $E_k \triangleq \Vert \hat{z}_{i,k} - z_{i} \Vert^2$, where $\hat{z}_{i,k}$ is the DLG attacker's estimate of agent $i$'s feature vector at time step $k$, and $z_{i}$ is the true feature vector of agent $i$. The parameters of RCP-SGD-5 are set as specified in Table \ref{tab:parameter}. It is then assumed that the DLG attacker's goal is to infer the feature vector of agent $2$. At each time step $k$, the agent obtains a noisy gradient $\tilde{\nabla} f_{i,k} = \nabla f_i(x_{i,k}) + \delta_s$, where $\delta_s \sim \mathcal{N}(0, 0.005)$ represents the noise. As shown in Figure \ref{fig:privacy1}, under the DSGD algorithm, the DLG attack can successfully estimate the feature vector of agent $2$, whereas under the RCP-SGD-5 algorithm, it fails to do so. This demonstrates that the RCP-SGD-5 algorithm can ensure the privacy of the agents' local data under the DLG attack.

\begin{figure}
	\centering
	\includegraphics[width=1\linewidth]{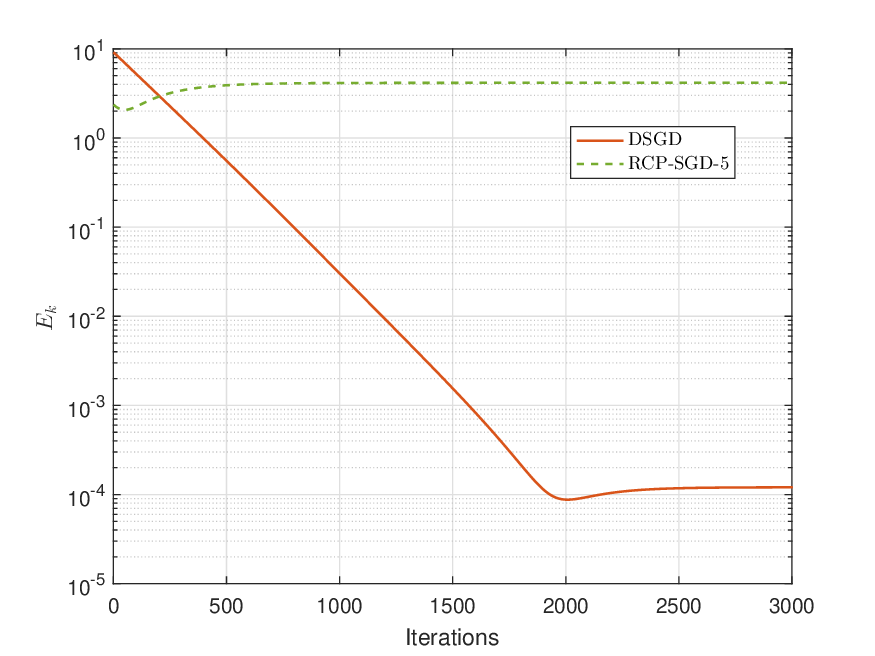}
	\caption{The evolution of estimate error of DLG uncer DSGD, and RCP-SGD-5}
	\label{fig:privacy1}
\end{figure}

\section{conclusion}\label{conclusion}
In this paper, we investigated distributed nonconvex optimization under limited communication with privacy concern. pecifically, we propose a robust compressed primal-dual SGD algorithm (RCP-SGD) that works for a broad class of compressors simultaneously exhibiting bounded absolute error and bounded relative error. For arbitrary smooth (possibly non-convex) objective functions, we proved that RCP-SGD achieves linear speedup convergence rate $\mathcal{O}(1/\sqrt{nT})$, where $T$ and $n$ are the number of iterations and agents, respectively. If the global cost function additionally satisfies the Polyak-Łojasiewicz condition, we proved that the proposed algorithm converge to the global optimum with a linear sppedup cpnvergence rate $\mathcal{O}(1/(nT))$. Notably, the optimal convergence rate for centralized stochastic optimization is known to be $\mathcal{O}(1/T)$. Finally, to exploit the inherent privacy benefits of compression, we introduce a simple yet effective transformation of the compressor. We proved that under the transformed compressors, RCP-SGD achieves rigorous differential privacy guarantees without any additional noise or degradation in convergence accuracy. Future work includes extending the study to directed graphs and online optimization.

\appendices
\section{Supporting Lemmas}
We first introduce some useful vector and matrix inequalities.
\begin{lemma}\label{lemma:lsmooth}
	\cite{liao2022compressed} Suppose the function $f(x): \R^d\mapsto\R$ is smooth with constant $L_f>0$, we have
 \begin{align*}
     \|\lf(x)\|^2\leq2L_f(f(x)-f^*).\addtag\label{eq:lsmooth}
 \end{align*}
	\end{lemma}
\begin{lemma}(Lemma 4 in~\cite{yi2022primal})
    For any constants $a\in(0,1)$, we have
    \begin{align*}
        (1-a)^T\leq\frac{k!}{(aT)^k},\forall k,T\in\mathbb{N}.\addtag\label{eq:linearbound}
    \end{align*}
\end{lemma}
\begin{lemma}
(Lemma 2 in~\cite{yi2022communication}) Suppose Assumption~\ref{as:strongconnected} holds, let $L$ be the Laplacian matrix of the graph $G$ and $K_n=\mathbf{I}_n-\frac{1}{n}\mathbf{1}_n\mathbf{1}_n^\top$. Then $L$ and $K_n$ are positive semi-definite, $L\leq\bar{\lambda}_L\mathbf{I}_n$, $\bar{\lambda}_{K_n}=1$, 
\begin{align}
&K_nL=LK_n=L,\label{eq:propertyofk}\\
&0\leq\underline{\lambda}_LK_n\leq L\leq\bar{\lambda}_LK_n.\label{eq:propertyofk1}
\end{align}
Moreover, there exists an orthogonal matrix $[r~R]\in\R^{n\times n}$ with $r=\frac{1}{\sqrt{n}}\mathbf{1}_n$ and $R\in\R^{n\times(n-1)}$ such that
\begin{align}
	&PL=LP=K_n,\label{eq:propertyofp}\\
	&\bar{\lambda}_L^{-1}\mathbf{I}_n\leq P\leq\underline{\lambda}_L^{-1}\mathbf{I}_n,\label{eq:propertyofp1}
\end{align}
where 
\begin{align*}
	P=\begin{bmatrix}
		r&R
	\end{bmatrix}
	\begin{bmatrix}
		\lambda_n^{-1}&0\\
		0&\Lambda_1^{-1}
	\end{bmatrix}
	\begin{bmatrix}
		r^\top\\
		R^\top\top
	\end{bmatrix},
\end{align*}
with $\Lambda_1=\diag([\lambda_2,\dots,\lambda_n])$ and $0<\lambda_2\leq\cdots\leq\lambda_n$ being the nonzero eigenvalues of $L$.
\end{lemma}
For simplicity of the proof, we denote some notations.
\begin{align*}
    &\x_k=[x_{1,k}^\top,\dots,x_{n,k}^\top]^\top\in\R^{nd} \\
    &\mathbf{g}^s_k=[\tilde{\nabla} f_{1,k}^\top,\dots,\tilde{\nabla} f_{n,k}^\top]\in\R^{nd}\\
    &\vv_k=[v_{1,k}^\top,\dots,v_{n,k}^\top]^\top\in\R^{nd}\\
    &\hx_k=\left[\hat{x}_{1,k}^\top,\dots,\hat{x}_{n,k}^\top\right]^\top\in\R^{nd}\\
    &\x_k^c=[{x^c}_{1,k}^\top,\dots,{x^c}_{n,k}^\top]^\top\in\R^{nd},
\end{align*}
we further denote $\bx_k=\mathbf{1}_n\otimes\bar{x}_k$, $\tilde{f}(\x_{k})=\sum_{i=1}^nf_i(x_{i,k}),~\bl=L\otimes \bi_d,~\bk=K\otimes \bi_d$, $\bp=P\otimes\bi_d$, $\bh=\frac{1}{n}({\mathbf 1}_n{\mathbf 1}_n^\top\otimes{\mathbf I}_d)$, $\g_k=\nabla\tilde{f}(\x_{k})$, $\bg_k=\bh\g_k$, $\g_k^b=\nabla\tilde{f}(\bx_{k})$, $\bg_k^b=\bh\g_k^b=\mathbf{1}_n\otimes\nabla f(\bar{x}_k)$. Then the update equations~\eqref{citerationx}--\eqref{eq:iterationv2} can be rewritten as the following compact form
\begin{align*}
		&~~~~\hx_{k}=\x_{k}^c+h_k\mathcal{C}((\x_{k}-\x_{k}^c)/h_k),\addtag\label{citerationx1}\\
	&~\x_{k+1}^c=(1-\alpha_x)\x_{k}^c+\alpha_x\hx_{k},\addtag\label{citerationxc1}\\
 &~\x_{k+1}=\x_{k}-\eta(\gamma \bl\hx_{k}+\omega  \vv_{k}+\mathbf{g}^s_k),\addtag\label{eq:iterationxp3}\\
	&~\vv_{k+1}=\vv_{k}+\eta\omega \bl\hx_{k}.\addtag\label{eq:iterationv3}
	\end{align*}
 From~\eqref{eq:iterationv3}, $\mathbf 1_n^\top L=L\mathbf 1_n=0$, and the fact that $\sum_{i=1}^n v_{i,0}=\mathbf{0}_d$, we have 
   \begin{align}\label{eq:propertyofv}
	\bv_{k+1}=\mathbf{0}_{nd}.
   \end{align}
Then from~\eqref{eq:iterationxp3}, $\mathbf 1_n^\top L=L\mathbf 1_n=0$, and~\eqref{eq:propertyofv}, one obtains that
\begin{align}\label{eq:propertyofbx}
	\bx_{k+1}=\bx_k-\eta\bg_k^s.
\end{align}
From Assumption~\ref{as:strongconnected}, one obtains that
\begin{align}\label{eq:propertyofbg1}
	\Vert\g^b_k-\g_k\Vert^2\leq L_f^2\Vert\bx_k-\x_k\Vert^2= L_f^2\Vert\x_k\Vert^2_\bk.
\end{align}
Furthermore, we have following useful equations
\begin{align*}
	&\Vert\bg^b_k-\bg_k\Vert^2=\Vert\bh(\g^b_k-\g_k)\Vert^2\leq L_f^2\Vert\x_k\Vert^2_\bk,\addtag\label{eq:propertyofbg2}\\
	&\Vert\g^b_{k+1}-\g^b_k\Vert^2\leq L_f^2\Vert\bx_{k+1}-\bx_k\Vert^2\\
	&~~~~~~~~~~~~~~~~~\leq \eta^2L_f^2\Vert\bg_k^s\Vert^2\addtag\label{eq:propertyofbg3},
\end{align*}
where the first inequality comes from~\eqref{eq:propertyofbg1} and $\lb_\bh=1$; the second inequality comes from Assumption~\ref{as:strongconnected}; the last inequality comes from~\eqref{eq:propertyofbx}. From Assumption~\ref{as:boundedvar}, we have
\begin{align*}
    &\E_{\xi_k}[\g^s_k]=\g_k,\addtag\label{eq:propertyofsg1}\\
    &\E_{\xi_k}[\Vert\g^s_k-\g_k\Vert^2]\leq n\sigma^2,\addtag\label{eq:propertyofsg2}\\
    &\E_{\xi_k}[\bg^s_k]=\E_{\xi_k}[\bh\g^s_k]=\bh\E_{\xi_k}[\g^s_k]=\bg_k.\addtag\label{eq:propertyofsg4}
\end{align*}
Combining~\eqref{eq:propertyofbg1},~\eqref{eq:propertyofsg2}, and Cauchy--Schwarz inequality, we have 
\begin{align*}
    \E_{\xi_k}[\Vert\g^s_k-\g_k^b\Vert^2]&\leq 2\E_{\xi_k}[\Vert\g^s_k-\g_k\Vert^2]+2\Vert\g^b_k-\g_k\Vert^2\\
    &\leq 2L_f^2\Vert\x_k\Vert^2_\bk+2n\sigma^2.\addtag\label{eq:propertyofsg3}
\end{align*}

\section{Auxiliary results}\label{app-convergence0}
We first construct some auxiliary functions and provide the following lemma
\begin{lemma}\label{lemma:lyapunov}
Suppose Assumptions~\ref{as:strongconnected}--\ref{as:finite}, and \ref{as:compressor} hold. Under Algorithm~\ref{Al:RCP-SGD}, if $\alpha_x\in(0,\frac{1}{r})$ and $\{\omega_k\}$ is non-decreasing, we have
\begin{align*}
    &\E_{\xi_k}[V_{1,k+1}]\leq V_{1,k}-\left\|\mathbf{x}_k\right\|_{\frac{\eta_k \gamma_k}{2} \mathbf{L}-\frac{\eta_k}{2} \mathbf{K}-\frac{\eta_k}{2}(1+5 \eta_k) L_f^2 \boldsymbol{K}}^2\\
    &~~~+\left\|\hat{\mathbf{x}}_k\right\|_{\frac{3\eta_k^2 \gamma_k^2}{2} \mathbf{L}^2}^2+ \frac{\eta_k}{2}(\gamma+2 \omega)\bar{\lambda}_L\left\|\hat{\mathbf{x}}_k-\mathbf{x}_k\right\|^2\\
&~~~ -\eta_k \omega_k \hat{\mathbf{x}}_k^{\top} \mathbf{K}\left(\mathbf{v}_k+\frac{1}{\omega_k} \g^b_k\right)\\
&~~~+\frac{6 \eta_k^2 \omega_k^2 \bar{\lambda}_L+\eta_k \omega_k}{4}\left\|\mathbf{v}_k+\frac{1}{\omega_k} \g^b_k\right\|^2_\bp+2n\sigma^2\eta_k^2,\addtag\label{eq:lyapunovofv1}\\
&\E_{\xi_k}[V_{2,k+1}]\leq V_{2,k}+\left(1+b_k\right)\eta_k\omega_k(1+\beta_1)\hx_k^\top\bk\left(\vv_k+\frac{1}{\omega_k}\g_k^b\right) \\
&~~~+\|\hx_k\|_{\left(1+b_k\right) \left(\frac{\eta_k^2\omega_k}{2}\left(\omega_k+\gamma_k\right) \bl+\frac{\eta_k^2}{2}\bk\right)}\\
&~~~+\frac{1}{2}\left(b_k+b_k\beta_1+\frac{\eta_k}{2}+\frac{b_k\eta_k}{2}\right)\left\|\vv_k+\frac{1}{\omega_k} \g_k^b\right\|_\bp^2\\
&~~~+(1+b_k)\bigg(\left(\frac{(1+\beta_1)^2}{\eta_k\omega_k^2}+\frac{1+\beta_1}{2\omega_k^2}\right)\frac{1}{\underline{\lambda}_L}\\
&~~~+\frac{1}{2}\bigg)\eta_k^2L_f^2\E_{\xi_k}\left[\Vert\bg_k^s\Vert^2\right]\\
&~~~+\E_{\xi_k}\left[(1+b_k)\eta_k\beta_1\hx_k^\top\bk(\g_{k+1}^b-\g_k^b)\right]\\
&~~~+\frac{1}{2\underline{\lambda}_L}(b_k+b_k^2)(1+\beta_1)\E_{\xi_k}\Vert\g_{k+1}^b\Vert^2,\addtag\label{eq:lyapunovofv2}\\
   &\E_{\xi_k}[V_{3,k+1}]\leq V_{3,k}-(1+b_k)\eta_k\gamma_k\hx_k^\top\bk(\vv_k+\frac{1}{\omega_k}\g_k^b)\\
   &~~~+\Vert\hx_k\Vert^2_{\eta_k(\omega_k\bk+\frac{b_k\gamma_k}{8}\bl)+\eta_k^2(\omega_k^2\bk+\frac{b_k}{2}\bk-\omega_k\gamma_k\bl)}\\
&~~~+\Vert\x_k\Vert_{(\frac{\eta_k(\omega_k+2)}{4})\bk+(\frac{\eta_k}{4}+\frac{3\eta_k^2}{2})L_f^2\bk}^2\\
&~~~-\E_{\xi_k}\left[(1+b_k)\frac{\eta_k\gamma_k}{\omega_k}\hx_k^\top\bk(\g_{k+1}^b-\g_k^b)+\frac{\eta_k}{8}\Vert\bg_k\Vert^2\right]\\
 &~~~-\Vert\vv_k+\frac{1}{\omega_k}\g_k^b\Vert_{\eta_k(\omega_k-3\underline{\lambda}_L^{-1})\bp-\eta_k^2(\underline{\lambda}_L^{-1}+\frac{\omega_k^2}{2}\bar{\lambda}_L)\bp-2b_k\eta_k\gamma_k\bp}^2\\
 &~~~+\frac{b_k}{2}(\|\x_{k+1}\|^2_\bk+\E_{\xi_k}\|\g_{k+1}^b\|^2)\\
 &~~~+(\frac{1+\eta_k}{2\eta_k\omega_k^2\underline{\lambda}_L^2}+\frac{b_k\gamma_k^2}{2\omega_k^2}+\frac{1}{4})\eta_k^2L_f^2\E_{\xi_k}[\Vert\bg_k^s\Vert^2]+n\eta_k^2\sigma^2,\addtag\label{eq:lyapunovofv3}\\
 &\E_{\xi_k}[V_{4,k+1}]\leq V_{4,k}-\frac{\eta_k}{4}\|\bg_k\|^2+\frac{\eta_kL_f^2}{2}\|\x_k\|_\bk^2-\frac{\eta_k}{4}\|\bg_k^b\|^2\\
 &~~~+\frac{\eta_k^2L_f}{2}\E_{\xi_k}\Vert\bg_k^s\Vert^2,\addtag\label{eq:lyapunovofv4}\\
 &\E_{\xi_k}[V_{5,k+1}]\leq (1-\frac{\varphi_1}{2}-\frac{\varphi_1^2}{2}\\
 &~~~+4\eta_k^2\gamma_k^2\bar{\lambda}_L^2r_0(1+\frac{2}{\varphi_1}))\Vert\x_k-\x_{k}^c\Vert^2\\
 &~~~+\Vert\x_k\Vert^2_{4\eta_k^2(1+\frac{2}{\varphi_1})(\gamma_k^2\bar{\lambda}_L^2+2L_f^2)\bk}\\
	&~~~+\Vert\vv_k+\frac{1}{\omega_k}\g_k^b\Vert^2_{4\eta_k^2(1+\frac{2}{\varphi_1})\omega_k^2\underline{\lambda}_L\bp}+(1+\frac{2}{\varphi_1})8n\eta_k^2\sigma^2\\
    &~~~+(8\eta^2\gamma^2\bar{\lambda}_L^2r^2+1)h_k^2\sigma_{\mathcal{C}},\addtag\label{eq:lyapunovofv5}
\end{align*}
where
\begin{align*}
&V_{1,k+1}=\frac{1}{2}\Vert\x_{k+1}\Vert_\bk^2\\
&V_{2,k+1}=\frac{1}{2}\Vert\vv_{k+1}+\frac{1}{\omega_{k+1}}\g_{k+1}^b\Vert^2_{\bp+\beta_1\bp}\\
&V_{3,k+1}=\x_{k+1}^\top\bk\bp(\vv_{k+1}+\frac{1}{\omega_{k+1}}\g_{k+1}^b)\\
    &V_{4,k+1}=n(f(\bar{x}_{k+1})-f^*),\\
    &V_{5,k+1}=\Vert\x_{k+1}-\x^c_{k+1}\Vert^2,\\
&~~~~~b_k=\frac{1}{\omega_k}-\frac{1}{\omega_{k+1}},\\
&~~~~~\varphi_1=\alpha_xr\varphi.
\end{align*}

\end{lemma}

\begin{proof}\textbf{(i)} This part shows the upper bound of $V_{1,k+1}$.
    \begin{align*}
\E_{\xi_k}[\frac{1}{2}\Vert&\x_{k+1}\Vert_\bk^2]=\E_{\xi_k}[\frac{1}{2}\Vert\x_{k}-\eta_k(\gamma_k \bl\hx_{k}+\omega_k  \vv_{k}+\g^s_k)\Vert_\bk^2]\\
& =\frac{1}{2}\left\|\mathbf{x}_k\right\|_{\mathbf{K}}^2-\eta_k \gamma_k \mathbf{x}_k^{\top} \mathbf{L} \hat{\mathbf{x}}_k+\left\|\hat{\mathbf{x}}_k\right\|_{\frac{\eta_k^2 \gamma_k^2}{2} \mathbf{L}^2}^2 \\
&~~~-\eta_k \omega_k\left(\mathbf{x}_k^{\top}-\eta_k \gamma_k \hat{\mathbf{x}}_k^{\top} \mathbf{L}\right) \mathbf{K}\left(\mathbf{v}_k+\frac{1}{\omega_k} \mathbf{g}_k\right)\\
&~~~+\E_{\xi_k}\left\|\mathbf{v}_k+\frac{1}{\omega_k} \mathbf{g}_k^s\right\|_{\frac{\eta_k^2 \omega_k^2}{2} \mathbf{K}}^2 \\
& =\frac{1}{2}\left\|\mathbf{x}_k\right\|_{\mathbf{K}}^2-\eta_k \gamma_k \mathbf{x}_k^{\top} \mathbf{L}\left(\mathbf{x}_k+\hat{\mathbf{x}}_k-\mathbf{x}_k\right)+\left\|\hat{\mathbf{x}}_k\right\|_{\frac{\eta_k^2 \gamma_k^2}{2}\mathbf{L}^2}^2 \\
&~~~-\eta_k \omega_k\left(\mathbf{x}_k^{\top}-\eta_k \gamma_k \hat{\mathbf{x}}_k^{\top} \mathbf{L}\right) \mathbf{K}\bigg(\mathbf{v}_k+\frac{1}{\omega_k} \g^b_k\\
&~~~+\frac{1}{\omega_k} \mathbf{g}_k-\frac{1}{\omega_k} \g^b_k\bigg) \\
& +\E_{\xi_k}\left\|\mathbf{v}_k+\frac{1}{\omega_k} \g^b_k+\frac{1}{\omega_k} \mathbf{g}_k^s-\frac{1}{\omega_k} \g^b_k\right\|_{\frac{\eta_k^2 \omega_k^2}{2} \mathbf{K}}^2 \\
& \leq \frac{1}{2}\left\|\mathbf{x}_k\right\|_{\mathbf{K}}^2-\left\|\mathbf{x}_k\right\|_{\eta_k\gamma_k\mathbf{L}}^2+\left\|\mathbf{x}_k\right\|_{\frac{\eta_k\gamma_k}{2} \mathbf{L}}^2\\
&~~~+\left\|\hat{\mathbf{x}}_k-\mathbf{x}_k\right\|_{\frac{\eta_k\gamma_k}{2} \mathbf{L}}^2 +\left\|\hat{\mathbf{x}}_k\right\|_{\frac{\eta_k^2 \gamma_k^2}{2} \mathbf{L}^2}^2 \\
&~~~-\eta_k \omega_k \mathbf{x}_k^{\top} \mathbf{K}\left(\mathbf{v}_k+\frac{1}{\omega_k} \g^b_k\right)+\frac{\eta_k}{2}\left\|\mathbf{x}_k\right\|_{\mathbf{K}}^2 \\
&~~~+\frac{\eta_k}{2}\left\|\mathbf{g}_k-\g^b_k\right\|^2+\left\|\hat{\mathbf{x}}_k\right\|_{\frac{\eta_k^2 \gamma_k^2}{2} \mathbf{L}^2}^2\\
&~~~+\frac{\eta_k^2 \omega_k^2}{2}\left\|\mathbf{v}_k+\frac{1}{\omega_k} \g^b_k\right\|^2 +\left\|\hat{\mathbf{x}}_k\right\|_{\frac{\eta_k^2 \gamma_k^2}{2} \mathbf{L}^2}^2\\
& ~~~+\frac{\eta_k^2}{2}\left\|\mathbf{g}_k-\g^b_k\right\|^2  +\eta_k^2 \omega_k^2\left\|\mathbf{v}_k+\frac{1}{\omega_k} \g^b_k\right\|^2\\
&~~~+\eta_k^2\E_{\xi_k}\left\|\mathbf{g}_k^s-\g^b_k\right\|^2 \\
& =\frac{1}{2}\left\|\mathbf{x}_k\right\|_{\mathbf{K}}^2-\left\|\mathbf{x}_k\right\|_{\frac{\eta_k \gamma_k}{2} \mathbf{L}-\frac{\eta_k}{2} \mathbf{K}}^2+\left\|\hat{\mathbf{x}}_k\right\|_{\frac{3 \eta_k^2 \gamma_k^2}{2} \mathbf{L}^2}^2 \\
& +\frac{\eta_k}{2}(1+ \eta_k)\left\|\mathbf{g}_k-\g^b_k\right\|^2+\left\|\hat{\mathbf{x}}_k-\mathbf{x}_k\right\|_{\frac{\eta_k\gamma_k}{2}\bl}^2 \\
& -\eta_k \omega_k\left(\hat{\mathbf{x}}_k+\mathbf{x}_k-\hat{\mathbf{x}}_k\right)^{\top} \mathbf{K}\left(\mathbf{v}_k+\frac{1}{\omega_k} \g^b_k\right)\\
&~~~+\frac{3 \eta_k^2 \omega_k^2}{2}\left\|\mathbf{v}_k+\frac{1}{\omega_k} \g^b_k\right\|^2+\eta_k^2\E_{\xi_k}\left\|\mathbf{g}_k^s-\g^b_k\right\|^2 \\
& \leq \frac{1}{2}\left\|\mathbf{x}_k\right\|_{\mathbf{K}}^2-\left\|\mathbf{x}_k\right\|_{\frac{\eta_k \gamma_k}{2} \mathbf{L}-\frac{\eta_k}{2} \mathbf{K}}^2+\left\|\hat{\mathbf{x}}_k\right\|_{\frac{3\eta_k^2\gamma_k^2}{2} \mathbf{L}^2}^2 \\
&~~~+\frac{\eta_k}{2}(1+\eta_k)\left\|\mathbf{g}_k-\g^b_k\right\|^2\\
&~~~+\left\|\hx_k-\x_k\right\|_{\frac{\eta_k}{2}(\gamma_k \mathbf{L}+2 \omega_k \bar{\lambda}_L \mathbf{K})}^2 \\
&~~~-\eta_k \omega_k \hat{\mathbf{x}}_k^{\top} \mathbf{K}\left(\mathbf{v}_k+\frac{1}{\omega_k} \g^b_k\right)+\eta_k^2\E_{\xi_k}\left\|\mathbf{g}_k^s-\g^b_k\right\|^2\\
&~~~+\frac{6 \eta_k^2 \omega_k^2+\eta_k \omega_k \bar{\lambda}_L^{-1}}{4}\left\|\mathbf{v}_k+\frac{1}{\omega_k} \g^b_k\right\|^2,\addtag\label{eq:upperboundofx1}
\end{align*}
where the first and second equalities comes from~\eqref{eq:propertyofk},~\eqref{eq:iterationxp3}, \eqref{eq:propertyofsg1} and Assumption~\ref{as:boundedvar}; the first inequality come from Cauchy--Schwarz inequality,~\eqref{eq:propertyofk}, and $\lb_\bk=1$; the second inequality comes from Cauchy--Schwarz inequality and $\lb_\bk=1$. From~\eqref{eq:propertyofk1},~\eqref{eq:propertyofp1},~\eqref{eq:propertyofbg1},~\eqref{eq:propertyofsg3} and~\eqref{eq:upperboundofx1}, we know that~\eqref{eq:lyapunovofv1} holds.

\textbf{(ii)} This part shows the upper bound of $V_{2,k+1}$. From the sequence $\{\omega_k\}$ is non-decreasing and~Cauchy–Schwarz inequality, one obtains that
\begin{align*}
    V_{2,k+1}&=\frac{1}{2}\Vert\vv_{k+1}+\frac{1}{\omega_{k+1}}\g_{k+1}^b\Vert^2_{\bp+\beta_1\bp}\\
    &=\frac{1}{2}\Vert\vv_{k+1}+\frac{1}{\omega_{k}}\g_{k+1}^b+(\frac{1}{\omega_{k+1}}-\frac{1}{\omega_{k}})\g_{k+1}^b\Vert^2_{\bp+\beta_1\bp}\\
    &\leq\frac{1}{2}(1+b_k)\Vert\vv_{k+1}+\frac{1}{\omega_{k}}\g_{k+1}^b\Vert^2_{\bp+\beta_1\bp}\\
    &~~~+\frac{1}{2}(b_k+b_k^2)\Vert\g_{k+1}^b\Vert^2_{\bp+\beta_1\bp}.\addtag\label{eq:propertyofv21}
\end{align*}
With respect to $\Vert\vv_{k+1}+\frac{1}{\omega_{k}}\g_{k+1}^b\Vert^2_{\bp+\beta_1\bp}$, we have
\begin{align*}
	\E_{\xi_k}&\bigg[\frac{1}{2}\Vert\vv_{k+1}+\frac{1}{\omega_k}\g_{k+1}^b\Vert^2_{\bp+\beta_1\bp}\bigg]\\
	&=\E_{\xi_k}\bigg[\frac{1}{2}\Vert\vv_k\!+\!\frac{1}{\omega_k}\g_k^b+\!\eta_k\omega_k\bl\hx_k\!+\!\frac{1}{\omega_k}(\g_{k+1}^b-\g_k^b)\Vert^2_{\bp+\beta_1\bp}\bigg]\\
	&= \frac{1}{2}\Vert\vv_k+\frac{1}{\omega_k}\g_k^b\Vert^2_{\bp+\beta_1\bp}\\
&~~~+\eta_k\omega_k(1+\beta_1)\hx_k^\top\bk\left(\vv_k+\frac{1}{\omega_k}\g_k^b\right)\\
  &~~~+\Vert\hx_k\Vert_{\frac{\eta_k^2\omega_k^2}{2}(1+\beta_1)\bl}^2+\E_{\xi_k}\bigg[\frac{1}{2\omega_k^2}\Vert \g_{k+1}^b-\g_k^b\Vert^2_{\bp+\beta_1\bp}\bigg]\\
	&~~~+\E_{\xi_k}\bigg[\frac{1}{\omega_k}(\g_{k+1}^b-\g_k^b)^\top(\bp+\beta_1\bp)\left(\vv_k+\frac{1}{\omega_k}\g_k^b\right)\bigg]\\
	&~~~+\E_{\xi_k}\bigg[\eta_k\hx_k^\top(\bk+\beta_1\bk)(\g_{k+1}^b-\g_k^b)\bigg]\\
	&\leq V_{2,k}+\eta_k\omega_k(1+\beta_1)\hx_k^\top\bk\left(\vv_k+\frac{1}{\omega_k}\g_k^b\right)\\
 &~~~+\Vert\hx_k\Vert_{\frac{\eta_k^2\omega_k^2}{2}(1+\beta_1)\bl}^2+\frac{1}{2\omega_k^2}\E_{\xi_k}\Vert \g_{k+1}^b-\g_k^b\Vert^2_{\bp+\beta_1\bp}\\
	&~~~+\Vert\vv_k+\frac{1}{\omega_k}\g_k^b\Vert^2_{\frac{\eta_k}{4}\bp}+\E_{\xi_k}\Vert\g_{k+1}^b-\g_k^b\Vert^2_{\frac{(1+\beta_1)^2}{\eta_k\omega_k^2}\bp}\\
	&~~~+\Vert\hx_k\Vert_{\frac{\eta_k^2}{2}\bk}^2\!+\!\frac{1}{2}\E_{\xi_k}\Vert\g_{k+1}^b-\g_k^b\Vert^2\\  &~~~+\E_{\xi_k}\left[\eta_k\beta_1\hx_k^\top\bk(\g_{k+1}^b-\g_k^b)\right]\\	&=V_{2,k}+\eta_k\omega_k(1+\beta_1)\hx_k^\top\bk\left(\vv_k+\frac{1}{\omega_k}\g_k^b\right)\\	&~~~+\Vert\hx_k\Vert^2_{\frac{\eta_k^2\omega_k^2}{2}(1+\beta_1)\bl+\frac{\eta_k^2}{2}\bk}+\Vert\vv_k+\frac{1}{\omega_k}\g_k^b\Vert^2_{\frac{\eta_k}{4}\bp}\\
	&~~~+\E_{\xi_k}\Vert\g_{k+1}^b-\g_k^b\Vert^2_{(\frac{(1+\beta_1)^2}{\eta_k\omega_k^2}+\frac{1+\beta_1}{2\omega_k^2})\bp}\\
    &~~~+\frac{1}{2}\E_{\xi_k}\Vert\g_{k+1}^b-\g_k^b\Vert^2\\	&~~~+\E_{\xi_k}\left[\eta_k\beta_1\hx_k^\top\bk(\g_{k+1}^b-\g_k^b)\right]\\
 &\leq V_{2,k}+\eta_k\omega_k(1+\beta_1)\hx_k^\top\bk\left(\vv_k+\frac{1}{\omega_k}\g_k^b\right)\\	&~~~+\Vert\hx_k\Vert^2_{\frac{\eta_k^2\omega_k^2}{2}(1+\beta_1)\bl+\frac{\eta_k^2}{2}\bk}+\Vert\vv_k+\frac{1}{\omega_k}\g_k^b\Vert^2_{\frac{\eta_k}{4}\bp}\\
	&~~~+\left(\left(\frac{(1+\beta_1)^2}{\eta_k\omega_k^2}+\frac{1+\beta_1}{2\omega_k^2}\right)\frac{1}{\underline{\lambda}_L}+\frac{1}{2}\right)\eta_k^2L_f^2\E_{\xi_k}[\Vert\bg_k^s\Vert^2]\\	&~~~+\E_{\xi_k}\left[\eta_k\beta_1\hx_k^\top\bk(\g_{k+1}^b-\g_k^b)\right],\addtag\label{eq:propertyofv22}
\end{align*}
where $\tilde{b}_1=(\frac{1}{\omega_k^2}+\frac{1}{2\eta_k\omega_k})(1+\frac{\gamma_k}{\omega_k})\frac{1}{\underline{\lambda}_L}+\frac{1}{2}$; the first equality comes from~\eqref{eq:iterationv3}; the second equality comes from~\eqref{eq:propertyofk} and~\eqref{eq:propertyofp}; the first inequality comes from~Cauchy–Schwarz inequality; the last inequality comes from~\eqref{eq:propertyofp1} and~\eqref{eq:propertyofbg3}. Combining~\eqref{eq:propertyofp1},~\eqref{eq:propertyofv21} and~\eqref{eq:propertyofv22}, we know that~\eqref{eq:lyapunovofv2} holds.

\textbf{(iii)} This part shows the upper bound of $V_{3,k+1}$. Similar to~\eqref{eq:propertyofv21}, one obtains that
\begin{align*}
V_{3,k+1}&=\x_{k+1}^\top\bk\bp(\vv_{k+1}+\frac{1}{\omega_{k+1}}\g_{k+1}^b)\\
&=\x_{k+1}^\top\bk\bp(\vv_{k+1}+\frac{1}{\omega_{k}}\g_{k+1}^b+(\frac{1}{\omega_{k+1}}-\frac{1}{\omega_{k}})\g_{k+1}^b)\\
&\leq \x_{k+1}^\top\bk\bp(\vv_{k+1}+\frac{1}{\omega_{k}}\g_{k+1}^b)\\
&~~~+\frac{b_k}{2}(\|\x_{k+1}\|^2_\bk+\|\g_{k+1}^b\|^2).\addtag\label{eq:propertyofv31}
\end{align*}
Regarding the first term of~\eqref{eq:propertyofv31}, it  holds that
\begin{align*}	
\E_{\xi_k}&[\x_{k+1}^\top\bk\bp(\vv_{k+1}+\frac{1}{\omega_k}\g_{k+1}^b)]\\
	&=\E_{\xi_k}[(\x_k-\eta_k(\gamma_k\bl\hx_k+\omega_k\vv_k+\g_k^b+\g_k^s-\g_k^b))^\top\\
	&~~~\bk\bp(\vv_k +\frac{1}{\omega_k}\g_k^b+\eta_k\omega_k\bl\hx_k+\frac{1}{\omega_k}(\g_{k+1}^b-\g_k^b))]\\
	&=(\x_k^\top\bk\bp-\eta_k(\gamma_k+\eta_k\omega_k^2)\hx_k^\top\bk)(\vv_k+\frac{1}{\omega_k}\g_k^b)\\
	&~~~+\eta_k\omega_k\x_k^\top\bk\hx_k-\Vert\hx_k\Vert^2_{\eta_k^2\gamma_k\omega_k\bl}\\
 &~~~+\frac{1}{\omega_k}\E_{\xi_k}\left[(\x_k^\top\bk\bp-\eta_k\gamma_k\hx_k^\top\bk)(\g_{k+1}^b-\g_k^b)\right]\\
 &~~~-\eta_k(\omega_k\vv_k\!+\!\g_k^b+\!\g_k\!-\!\g_k^b\!-\bg_k)^\top\bp(\vv_k+\frac{1}{\omega_k}\g_k^b)\\
 &~~~-\E_{\xi_k}\left[\eta_k(\vv_k+\frac{1}{\omega_k}\g_k^b)^\top\bp\bk(\g_{k+1}^b-\g_k^b)\right]\\
	&~~~-\E_{\xi_k}[\eta_k(\g_k^s-\g_k^b)^\top(\eta_k\omega_k\bk\hx_k\\
 &~~~+\frac{1}{\omega_k}\bk\bp(\g_{k+1}^b-\g_k^b))]\\
	&\leq(\x_k^\top\bk\bp-\eta_k\gamma_k\hx_k^\top\bk)(\vv_k+\frac{1}{\omega_k}\g_k^b)+\Vert\hx_k\Vert_{\frac{\eta_k^2\omega_k^2}{2}\bk}^2\\
	&~~~+\Vert\vv_k+\frac{1}{\omega_k}\g_k^b\Vert_{\frac{\eta_k^2\omega_k^2}{2}}^2+\Vert\x_k\Vert_{\frac{\eta_k\omega_k}{4}\bk}^2\\
 &~~~+\Vert\hx_k\Vert_{\eta_k\omega_k(\bk-\eta_k\gamma_k\bl)}^2+\Vert\x_k\Vert_{\frac{\eta_k}{2}\bk}^2\\
	&~~~+\E_{\xi_k}\left[\Vert\g_{k+1}^b-\g_k^b\Vert^2_{\frac{1}{2\eta_k\omega_k^2}\bp^2}\right]\\
	&~~~-\E_{\xi_k}\left[\frac{\eta_k\gamma_k}{\omega_k}\hx_k^\top\bk(\g_{k+1}^b-\g_k^b)\right]-\Vert\vv_k+\frac{1}{\omega_k}\g_k^b\Vert_{\eta_k\omega_k\bp}^2\\
	&~~~+\frac{\eta_k}{4}\Vert\g_k-\g_k^b\Vert^2+\frac{\eta_k}{8}\Vert\bg_k\Vert^2\\
	&~~~+\Vert\vv_k+\frac{1}{\omega_k}\g_k^b\Vert_{3\eta_k\bp^2}^2+\Vert\vv_k+\frac{1}{\omega_k}\g_k^b\Vert_{\eta_k^2\bp^2}^2\\
	&~~~+\frac{1}{4}\E_{\xi_k}\Vert\g_{k+1}^b-\g_k^b\Vert^2+\frac{\eta_k^2}{2}\Vert\g_{k}-\g_k^b\Vert^2+\Vert\hx_k\Vert_{\frac{\eta_k^2\omega_k^2}{2}\bk}^2\\
	&~~~+\E_{\xi_k}[\frac{\eta_k^2}{2}\Vert\g_{k}^s-\g_k^b\Vert^2]+\E_{\xi_k}[\Vert\g_{k+1}^b-\g_k^b\Vert_{\frac{1}{2\omega_k^2}\bp^2}^2]\\
 &=(\x_k^\top\bk\bp-\eta_k\gamma_k\hx_k^\top\bk)(\vv_k+\frac{1}{\omega_k}\g_k^b)\\	&~~~+\Vert\x_k\Vert_{\frac{\eta_k(\omega_k+2)}{4}\bk}^2+\Vert\hx\Vert^2_{\eta_k\omega_k\bk+\eta_k^2(\omega_k^2\bk-\omega_k\gamma_k\bl)}\\
	&~~~+(\frac{\eta_k}{4}+\frac{\eta_k^2}{2})\Vert\g_{k}-\g_k^b\Vert^2\\
 &~~~-\E_{\xi_k}\left[\frac{\eta_k\gamma_k}{\omega_k}\hx_k^\top\bk(\g_{k+1}^b-\g_k^b)\right]+\frac{\eta_k}{8}\Vert\bg_k\Vert^2\\
 &~~~-\Vert\vv_k+\frac{1}{\omega_k}\g_k^b\Vert_{\eta_k\omega_k\bp-3\eta_k\bp^2-\eta_k^2\bp^2-\frac{\eta_k^2\omega_k^2}{2}\bi_{nd}}^2\\
 &~~~+\E_{\xi_k}[\frac{\eta_k^2}{2}\Vert\g_{k}^s-\g_k^b\Vert^2]\\
	&~~~+\E_{\xi_k}[\Vert\g_{k+1}^b-\g_k^b\Vert^2_{\frac{1+\eta_k}{2\eta_k\omega_k^2}\bp^2+\frac{1}{4}\bi_{nd}}]\\
	&\leq V_{3,k}-(1+b_k)\eta_k\gamma_k\hx_k^\top\bk(\vv_k+\frac{1}{\omega_k}\g_k^b)\\ &~~~+\Vert\hx\Vert^2_{\eta_k\omega_k\bk+\eta_k^2(\omega_k^2\bk-\omega_k\gamma_k\bl)}\\	&~~~+\Vert\x_k\Vert_{\frac{\eta_k(\omega_k+2)}{4}\bk+(\frac{\eta_k}{4}+\frac{\eta_k^2}{2})L_f^2\bk}^2\\
  &~~~-\E_{\xi_k}\left[(1+b_k)\frac{\eta_k\gamma_k}{\omega_k}\hx_k^\top\bk(\g_{k+1}^b-\g_k^b)\right]+\frac{\eta_k}{8}\Vert\bg_k\Vert^2\\
 &~~~-\Vert\vv_k+\frac{1}{\omega_k}\g_k^b\Vert_{\eta_k(\omega_k-3\underline{\lambda}_L^{-1})\bp-\eta_k^2(\underline{\lambda}_L^{-1}-\frac{\omega_k^2}{2}\bar{\lambda}_L)\bp}^2\\
 &~~~+\E_{\xi_k}[\frac{\eta_k^2}{2}\Vert\g_{k}^s-\g_k^b\Vert^2]\\
	&~~~+\E_{\xi_k}[\Vert\g_{k+1}^b-\g_k^b\Vert^2_{\frac{1+\eta_k}{2\eta_k\omega_k^2\underline{\lambda}_\bl^2}\bi_{nd}+\frac{1}{4}\bi_{nd}}]\\
 &~~~+b_k\eta_k\gamma_k\hx_k^\top\bk(\vv_k+\frac{1}{\omega_k}\g_k^b)\\
	&~~~+\E_{\xi_k}\left[b_k\frac{\eta_k\gamma_k}{\omega_k}\hx_k^\top\bk(\g_{k+1}^b-\g_k^b)\right]\\
 &\leq V_{3,k}-(1+b_k)\eta_k\gamma_k\hx_k^\top\bk(\vv_k+\frac{1}{\omega_k}\g_k^b)\\ &~~~+\Vert\hx\Vert^2_{\eta_k(\omega_k\bk+\frac{b_k\gamma_k}{8}\bl)+\eta_k^2(\omega_k^2\bk+\frac{b_k}{2}\bk-\omega_k\gamma_k\bl)}\\
 &~~~+\Vert\x_k\Vert_{(\frac{\eta_k(\omega_k+2)}{4})\bk+(\frac{\eta_k}{4}+\frac{3\eta_k^2}{2})L_f^2\bk}^2\\
  &~~~-\E_{\xi_k}\left[(1+b_k)\frac{\eta_k\gamma_k}{\omega_k}\hx_k^\top\bk(\g_{k+1}^b-\g_k^b)\right]+\frac{\eta_k}{8}\Vert\bg_k\Vert^2\\
 &~~~-\Vert\vv_k+\frac{1}{\omega_k}\g_k^b\Vert_{\eta_k(\omega_k-3\underline{\lambda}_L^{-1})\bp-\eta_k^2(\underline{\lambda}_L^{-1}+\frac{\omega_k^2}{2}\bar{\lambda}_L)\bp-2b_k\eta_k\gamma_k\bp}^2\\
 &~~~+(\frac{1+\eta_k}{2\eta_k\omega_k^2\underline{\lambda}_L^2}+\frac{b_k\gamma_k^2}{2\omega_k^2}+\frac{1}{4})\eta_k^2L_f^2\E_{\xi_k}[\Vert\bg_k^s\Vert^2]+n\eta_k^2\sigma^2,\addtag\label{eq:propertyofv32}
\end{align*}
where the first equality comes from~\eqref{eq:iterationxp3} and~\eqref{eq:iterationv3}; the second equality holds due to~\eqref{eq:propertyofk},~\eqref{eq:propertyofp},~\eqref{eq:propertyofv},~\eqref{eq:propertyofsg1}, and the fact that $\bk=\bi-\bh$; the first inequality comes from~Cauchy–Schwarz inequality; the second inequality holds due to~\eqref{eq:propertyofp1} and~\eqref{eq:propertyofbg1}; the last inequality holds due to~Cauchy–Schwarz inequality,~\eqref{eq:propertyofbg3}, and~\eqref{eq:propertyofsg3}. Combining~\eqref{eq:propertyofv31} and~\eqref{eq:propertyofv32}, we know that~\eqref{eq:lyapunovofv3} holds. 

\textbf{(iv)} This part shows the upper bound of $V_{4,k+1}$.
\begin{align*}
	\E_{\xi_k}&[V_{4,k+1}]=n(f(\bar{x}_{k+1})-f^*)\\
 &=\tilde{f}(\bx_k)-nf^*+\tilde{f}(\bx_{k+1})-\tilde{f}(\bx_k)\\
	&\leq \tilde{f}(\bx_k)-nf^*-\E_{\xi_k}\left[\eta_k(\bg_k^s)^\top\bg_k^b\right]+\frac{\eta_k^2L_f}{2}\E_{\xi_k}\Vert\bg_k^s\Vert^2\\
	&=\tilde{f}(\bx_k)-nf^*-\frac{\eta_k}{2}\bg_k^\top(\bg_k^b+\bg_k-\bg_k)\\
	&~~~-\frac{\eta_k}{2}(\bg_k-\bg_k^b+\bg_k^b)^\top(\bg_k^b)+\frac{\eta_k^2L_f}{2}\E_{\xi_k}\Vert\bg_k^s\Vert^2\\
 &\leq\tilde{f}(\bx_k)-nf^*-\frac{\eta_k}{4}\|\bg_k\|^2+\frac{\eta_k}{2}\|\bg_k^b-\bg_k\|^2\\
	&~~~-\frac{\eta_k}{4}\|\bg_k^b\|^2+\frac{\eta_k^2L_f}{2}\E_{\xi_k}\Vert\bg_k^s\Vert^2\\
 &\leq\tilde{f}(\bx_k)-nf^*-\frac{\eta_k}{4}\|\bg_k\|^2+\frac{\eta_kL_f^2}{2}\|\x_k\|_\bk^2\\
	&~~~-\frac{\eta_k}{4}\|\bg_k^b\|^2+\frac{\eta_k^2L_f}{2}\E_{\xi_k}\Vert\bg_k^s\Vert^2,\addtag\label{eq:upperboundofax1}
\end{align*}
where the first inequality comes from~\eqref{eqn:smooth1},~\eqref{eq:propertyofbx}, and the fact that $\bh=\bh\bh$; the third equality holds due to~\eqref{eq:propertyofsg4}; the second inequality holds due to Cauchy--Schwarz inequality. From~\eqref{eq:propertyofbg2} and~\eqref{eq:upperboundofax1}, we know that~\eqref{eq:lyapunovofv4} holds.

\textbf{(v)} This part shows the upper bound of $V_{5,k+1}$
\begin{align*}
    \E&[V_{5,k+1}]=\E\|\x_{k+1}-\x_{k+1}^c\|^2\\
&=\E\Vert\x_{k+1}-\x_k+\x_k-\x_{k}^c-\alpha_xr\frac{h_k\mathcal{C}((\x_k-\x_{k}^c)/h_k)}{r}\Vert^2\\
&\leq (1+s)\E\Vert(1-\alpha_xr)(\x_k-\x_{k}^c)+\alpha_xrh_k((\x_k-\x_{k}^c)/h_k\\
&~~~-\frac{\mathcal{C}((\x_k-\x_{k}^c)/h_k)}{r})\Vert^2+(1+\frac{1}{s})\E\Vert\x_{k+1}-\x_k\Vert^2\\
&\leq (1+s)(\alpha_xr(1-\varphi)+(1-\alpha_xr))\E\Vert\x_k-\x_{k}^c\Vert^2\\
&~~~+(1+\frac{1}{s})\E\Vert\x_{k+1}-\x_k\Vert^2+h_k^2\sigma_{\mathcal{C}}\\
&= (1-\frac{\varphi_1}{2}-\frac{\varphi_1^2}{2})\E\Vert\x_k-\x_{k}^c\Vert^2\\
&~~~+(1+\frac{2}{\varphi_1})\E\Vert\x_{k+1}-\x_k\Vert^2+h_k^2\sigma_{\mathcal{C}},\addtag\label{eq:upperboundofc1}
\end{align*}
where the second equality comes from~\eqref{citerationx1} and~\eqref{citerationxc1}; the first inequality comes from AM--GM inequality and $s>0$; the second inequality comes from the convexity of the norm and~\eqref{eq:propertyofcompressors}; the last equality follows by denoting $\varphi_1=\alpha_xr\varphi$, choosing $s=\frac{\varphi_1}{2}$, and $\alpha_xr<1$. Regarding the term $\Vert\x_{k+1}-\x_k\Vert^2$, we have
\begin{align*}
	\E\Vert\x_{k+1}&-\x_k\Vert^2\\
 &=\E\|\eta\left(\gamma\bl\hx_k+\omega\vv_k+\g_k^s\right)\|^2\\
	&=\eta^2\E\Vert(\gamma\bl(\hx_k-\x_k)+\gamma\bl\x_k+\omega\vv_k+\g_k^b\\
	&~~~+\g_k^s-\g_k^b)\Vert^2\\
	&\leq 4\eta^2(\E\Vert\gamma\bl(\hx_k-\x_k)\Vert^2+\E\left\Vert\omega\vv_k+\g_k^b\right\Vert^2\\
	&~~~+\E\Vert\gamma\bl\x_k\Vert^2+\E\Vert\g_k^s-\g_k^b\Vert^2)\\
	&\leq 4\eta^2(\gamma^2\bar{\lambda}_L^2r_0\E\Vert\x^c_k-\x_k\Vert^2+\E\left\Vert\vv_k\!+\!\frac{1}{\omega}\g_k^b\right\Vert^2_{\omega^2\underline{\lambda}_L\bp}\\
&~~~+\E\Vert\x_k\Vert^2_{(\gamma^2\bar{\lambda}_L^2+2L_f^2)\bk}+2n\sigma^2)\\
&~~~+8\eta^2\gamma^2\bar{\lambda}_L^2r^2h_k^2\sigma_{\mathcal{C}},\addtag\label{eq:upperboundofc2}
\end{align*}
where the first equality holds due to~\eqref{eq:iterationxp3}; the first inequality holds due to Cauchy--Schwarz inequality; the last inequality holds due to~\eqref{eq:propertyofcompressors1},~\eqref{eq:propertyofk1},~\eqref{eq:propertyofp1}, and~\eqref{eq:propertyofsg3}. Combining~\eqref{eq:upperboundofc1} and \eqref{eq:upperboundofc2}, one obtains that
\begin{align*}
	\E[V_{5,k+1}]
	&\leq (1-\frac{\varphi_1}{2}-\frac{\varphi_1^2}{2}\\
	&~~~+4\eta^2\gamma^2\bar{\lambda}_L^2r_0(1+\frac{2}{\varphi_1}))\E\Vert\x_k-\x_{k}^c\Vert^2\\
 &~~~+\E\Vert\x_k\Vert^2_{4\eta^2(1+\frac{2}{\varphi_1})(\gamma^2\bar{\lambda}_L^2+2L_f^2)\bk}\\
	&~~~+\E\left\Vert\vv_k+\frac{1}{\omega}\g_k^b\right\Vert^2_{4\eta^2(1+\frac{2}{\varphi_1})\omega^2\underline{\lambda}_L\bp}\\&~~~+(1+\frac{2}{\varphi_1})8n\eta^2\sigma^2\\
&~~~+(8\eta^2\gamma^2\bar{\lambda}_L^2r^2+1)h_k^2\sigma_{\mathcal{C}},\addtag\label{eq:upperboundofc3}
\end{align*}
\end{proof}

\section{The proof of Theorem~\ref{theo:convergence1}}\label{app-convergence1}
For simplicity of the proof, we also denote some notations.
\begin{align*}
&\varphi_1=\alpha_xr\varphi\\
&\epsilon_1=\frac{\gamma}{2}\underline{\lambda}_L-(\frac{\omega+4}{4}+\frac{5}{4}L_f^2)\\
&\epsilon_2=(7+\frac{16}{\varphi_1})L_f^2+(4+\frac{8}{\varphi_1})\gamma^2\bar{\lambda}_L^2\\
&\epsilon_3=\frac{1}{2}+\omega^2+\frac{3\gamma^2\bar{\lambda}_L^2}{2}\\
&\epsilon_4=\frac{3\omega-1}{4}-3\underline{\lambda}_L^{-1}\\
&\epsilon_5=2\omega^2\bar{\lambda}_L+\underline{\lambda}_L^{-1}+(4+\frac{8}{\varphi_1})\omega^2\underline{\lambda}_L\\
&\epsilon_6=\frac{1}{8}-(\frac{2(1+\beta_1)^2}{\omega^2\underline{\lambda}_L}+\frac{1}{\omega^2\underline{\lambda}_L^2})L_f^2\\
&\epsilon_7=(\frac{1+\beta_1}{\omega^2\underline{\lambda}_L}+\frac{1}{\omega^2\underline{\lambda}_L^2}+\frac{3}{2})L_f^2+L_f\\
&\epsilon_8=(\frac{2(1+\beta_1)^2}{\eta\omega^2\underline{\lambda}_L}+\frac{1+\beta_1}{\omega^2\underline{\lambda}_L}+\frac{1+\eta}{\eta\omega^2\underline{\lambda}_L^2}+\frac{3}{2})L_f^2+L_f\\
&\epsilon_9=11+\frac{16}{\varphi_1}\\
&\epsilon_{10}=\frac{\varphi_1}{2}+\frac{\varphi_1^2}{2}\\
&\epsilon_{11}=\frac{1}{2}(\gamma+2 \omega)\bar{\lambda}_Lr_0+2\omega r_0\\
&\epsilon_{12}=\frac{(8+7\varphi_1)\gamma^2\bar{\lambda}_L^2r_0}{\varphi_1}+(1+2\omega^2)r_0\\
&\epsilon_{13}=(\beta_1\beta_2+6 \beta_2)\bar{\lambda}_Lr^2+(14\beta_1^2\beta_2^2\bar{\lambda}_L^2r^2+2\eta^2+4\beta_1^2\beta_2^2)+1\\
&\tilde{\epsilon}_1=\frac{\gamma}{2}\underline{\lambda}_L-(\frac{9\omega+4}{4}+\frac{5}{4}L_f^2)\\
&\tilde{\epsilon}_2=(7+\frac{16}{\varphi_1})L_f^2+(4+\frac{8}{\varphi_1})\gamma^2\bar{\lambda}_L^2+1+2\omega^2+3\gamma^2\bar{\lambda}_L^2\\
&\beta_3=\max\{\frac{4+5L_f^2}{\beta_5},\frac{12\underline{\lambda}_L^{-1}+1}{3},\sqrt{\beta_6},\frac{\beta_2}{\beta_4},4\beta_2 L_f\}\\
&\beta_4=\min\{\frac{\tilde{\epsilon}_1}{\tilde{\epsilon}_2},\frac{\epsilon_4}{\epsilon_5},\frac{\epsilon_6}{\epsilon_7},\frac{\sqrt{\epsilon_{11}^2+4\epsilon_{10}\epsilon_{12}}-\epsilon_{11}}{2\epsilon_{12}},1\}\\
&\beta_5>0\\
&\beta_6=(\frac{16(1+\beta_1)^2}{\underline{\lambda}_L}+\frac{8}{\underline{\lambda}_L^2})L_f^2\\
&\check{c}_1=\frac{\gamma \underline{\lambda}_L-\omega}{2 \gamma \underline{\lambda}_L}\\
&c_0=\max\{\frac{9+\beta_5}{2\underline{\lambda}_L},1\}\\
&c_1=(\frac{2(1+\beta_1)^2}{\beta_2\beta_3\underline{\lambda}_L}+\frac{1+\beta_1}{\beta_3^2\underline{\lambda}_L}+\frac{1}{\beta_2\beta_3\underline{\lambda}_L^2}\\
&~~~~~~~~+\frac{1}{\beta_3^2\underline{\lambda}_L^2}+\frac{3}{2})L_f^2+L_f\\
&c_2=\eta\tilde{\epsilon}_1-\eta^2\tilde{\epsilon}_2.
\end{align*}
where $\beta_1$ and $\beta_2$ are the parameters used in Theorem~\ref{theo:convergence1}; $\alpha_x$, $\omega$ and $\gamma$ are the parameters of the proposed algorithm (see~\eqref{citerationx}--\eqref{eq:iterationv2}).
\begin{lemma}\label{lemma:fixedinquali}
 Suppose Assumptions~\ref{as:strongconnected}--\ref{as:finite} and~\ref{as:compressor} hold. If $\gamma_k=\gamma=\beta_1\omega$, $\beta_1>1$, $\omega_k=\omega$, and $\alpha_x\in(0,\frac{1}{r})$, it  holds that
    \begin{align*}
        \E[V_{k+1}]&\leq \E[V_k]-\E\|\x_k\|_{(\eta\tilde{\epsilon}_1-\eta^2\tilde{\epsilon}_2)\bk}\\
        &~~~-\E\left\Vert\vv_k+\frac{1}{\omega}\g_k^b\right\Vert^2_{\eta(\epsilon_4-\eta\epsilon_5)\bp}-\eta(\epsilon_6-\eta\epsilon_7)\E\|\bg_k\|^2\\
 &~~~-\frac{\eta}{4}\E\|\bg_k^b\|^2+\epsilon_8\sigma^2\eta^2+\epsilon_9n\sigma^2\eta^2\\
 &~~~-(\epsilon_{10}-\eta\epsilon_{11}-\eta^2\epsilon_{12})\E\Vert\x_k-\x_{k}^c\Vert^2+\epsilon_{13}h_k^2\sigma_{\mathcal{C}},\addtag\label{eq:upperofV14}
    \end{align*}
    where $V_{k}=\sum_{i=1}^5 V_{i,k}$.
\end{lemma}
\begin{proof}
   
We first consider the term $\E_{\xi_k}[\Vert\bg_k^s\Vert^2]$.
\begin{align*}
    \E_{\xi_k}[\Vert\bg_k^s\Vert^2]&= \E_{\xi_k}[\Vert\bg_k^s-\bg_k+\bg_k\Vert^2]\\
    &\leq 2\E_{\xi_k}[\Vert\bg_k^s-\bg_k\Vert^2]+2\|\bg_k\|^2\\
    &=\frac{2}{n}\E_{\xi_k}[\Vert \sum_{i=1}^n g_{i,k}^s-g_{i,k}\Vert^2]+2\|\bg_k\|^2\\
    &=\frac{2}{n}\sum_{i=1}^n \E_{\xi_k}[\Vert g_{i,k}^s-g_{i,k}\Vert^2]+2\|\bg_k\|^2\\
    &\leq 2\sigma^2+2\|\bg_k\|^2,\addtag\label{eq:upperofV11}
\end{align*}
where the first inequality holds due to Cauchy--Schwarz inequality; the last equality holds due to Assumption~\ref{as:boundedvar}; the last inequality holds due to~\eqref{eq:propertyofsg2}. We then consider the term $\|\hx_k\|_\bk^2$.
\begin{align*}
   \|\hx_k\|_\bk^2&=\|\hx_k-\x_k+\x_k\|_\bk^2\\
   &\leq 2\|\hx_k-\x_k\|^2+2\|\x_k\|_\bk^2\\
   &\leq 2r_0\Vert\x_k-\x_{k}^c\Vert^2+2\|\x_k\|_\bk^2+4r^2h_k^2\sigma_{\mathcal{C}},\addtag\label{eq:upperofV12}
\end{align*}
where the last inequality holds due to~\eqref{eq:propertyofcompressors0}.

 From~\eqref{eq:propertyofcompressors0},~\eqref{eq:upperboundofc3},~\eqref{eq:upperofV11}, $b_k=0$, and Lemma~\ref{lemma:lyapunov}, we have
    \begin{align*}
\E&[V_{k+1}]\\
&\leq \E[V_k]-\E\left\|\mathbf{x}_k\right\|_{\frac{\eta \gamma}{2} \mathbf{L}-\frac{\eta}{2} \mathbf{K}-\frac{\eta}{2}(1+5 \eta) L_f^2 \boldsymbol{K}}^2+\E\left\|\hat{\mathbf{x}}_k\right\|_{\frac{3\eta^2 \gamma^2}{2} \mathbf{L}^2}^2 \\
	&~~~+\frac{\eta}{2}(\gamma+2 \omega)\bar{\lambda}_Lr_0\E\left\|\hat{\mathbf{x}}_k-\mathbf{x}^c_k\right\|^2 \\
&~~~ +\frac{6 \eta^2 \omega^2 \bar{\lambda}_L+\eta \omega+\eta}{4}\E\left\|\mathbf{v}_k+\frac{1}{\omega} \g^b_k\right\|^2_\bp+2n\sigma^2\eta^2\\
	&~~~+\E\|\hx_k\|_{ \frac{\eta^2\omega}{2}\left(\omega+\gamma\right) \bl+\frac{\eta^2}{2}\bk}\\
&~~~+\left(\left(\frac{(1+\beta_1)^2}{\eta\omega^2}+\frac{1+\beta_1}{2\omega^2}\right)\frac{1}{\underline{\lambda}_L}+\frac{1}{2}\right)\\
	&~~~\eta^2L_f^2(2\sigma^2+2\E\|\bg_k\|^2)\\
&~~~+\E\Vert\hx\Vert^2_{\eta\omega\bk+\eta^2(\omega^2\bk-\omega\gamma\bl)}\\
	&~~~+\E\Vert\x_k\Vert_{(\frac{\eta(\omega+2)}{4})\bk+(\frac{\eta}{4}+\frac{3\eta^2}{2})L_f^2\bk}^2+\frac{\eta}{8}\E\Vert\bg_k\Vert^2\\
&~~~-\E\left\Vert\vv_k+\frac{1}{\omega}\g_k^b\right\Vert_{\eta(\omega-3\underline{\lambda}_L^{-1})\bp-\eta^2(\underline{\lambda}_L^{-1}+\frac{\omega^2}{2}\bar{\lambda}_L)\bp}^2\\
&~~~+(\frac{1+\eta}{\eta\omega^2\underline{\lambda}_L^2}+\frac{1}{2})\eta^2L_f^2(\sigma^2+\E\|\bg_k\|^2)+n\eta^2\sigma^2\\
&~~~-\frac{\eta}{4}\E\|\bg_k\|^2+\frac{\eta L_f^2}{2}\E\|\x_k\|_\bk^2-\frac{\eta}{4}\E\|\bg_k^b\|^2\\
	&~~~+\eta^2L_f(\sigma^2+\E\|\bg_k\|^2)\\
&~~~+(-\frac{\varphi_1}{2}-\frac{\varphi_1^2}{2}+4\eta^2\gamma^2\bar{\lambda}_L^2r_0(1+\frac{2}{\varphi_1}))\E\Vert\x_k-\x_{k}^c\Vert^2\\
	&~~~+\E\Vert\x_k\Vert^2_{4\eta^2(1+\frac{2}{\varphi_1})(\gamma^2\bar{\lambda}_L^2+2L_f^2)\bk}\\
	&~~~+\E\left\Vert\vv_k+\frac{1}{\omega}\g_k^b\right\Vert^2_{4\eta^2(1+\frac{2}{\varphi_1})\omega^2\underline{\lambda}_L\bp}+(1+\frac{2}{\varphi_1})8n\eta^2\sigma^2\\
    &~~~+(8\eta^2\gamma^2\bar{\lambda}_L^2r^2+1)h_k^2\sigma_{\mathcal{C}}\\
 &\leq \E[V_k]-\E\|\x_k\|_{(\eta\epsilon_1-\eta^2\epsilon_2)\bk}+\E\|\hx_k\|^2_{\eta\omega\bk+\eta^2\epsilon_3\bk}\\
	&~~~-\E\left\Vert\vv_k+\frac{1}{\omega}\g_k^b\right\Vert^2_{\eta(\epsilon_4-\eta\epsilon_5)\bp}-\eta(\epsilon_6-\eta\epsilon_7)\E\|\bg_k\|^2\\
 &~~~-\frac{\eta}{4}\E\|\bg_k^b\|^2+\epsilon_8\sigma^2\eta^2+\epsilon_9n\sigma^2\eta^2\\
 &~~~+(-\frac{\varphi_1}{2}-\frac{\varphi_1^2}{2}+4\eta^2\gamma^2\bar{\lambda}_L^2r_0(1+\frac{2}{\varphi_1})\\
	&~~~+\frac{\eta}{2}(\gamma+2 \omega)\bar{\lambda}_Lr_0)\E\Vert\x_k-\x_{k}^c\Vert^2\\
    &~~~+(\eta(\gamma+2 \omega)\bar{\lambda}_Lr^2+8\eta^2\gamma^2\bar{\lambda}_L^2r^2+1)h_k^2\sigma_{\mathcal{C}},\addtag\label{eq:upperofV13}
    \end{align*}
    where the second inequality due to~\eqref{eq:propertyofcompressors1},~\eqref{citerationx1},~\eqref{eq:propertyofk1} and $\beta_1>1$. Combining~\eqref{eq:upperofV12} and~\eqref{eq:upperofV13}, we complete the proof.
\end{proof}

Then we ready to prove Theorem~\ref{theo:convergence1}

\textbf{(i)} From $\gamma=\beta_1\omega$, $\beta_1>\frac{9+\beta_5}{2\underline{\lambda}_L}$, $\beta_5>0$, and $\omega>\beta_3\geq\frac{4+5L_f^2}{\beta_5}$ we have
\begin{align*}
    \tilde{\epsilon}_1&=\frac{\beta_1\omega}{2}\underline{\lambda}_L-(\frac{9\omega+4}{4}+\frac{5}{4}L_f^2)\\
    &>\frac{\beta_1\omega}{2}\underline{\lambda}_L-\frac{(9+\beta_5)\omega}{4}>0.\addtag\label{eqn:bound1}
\end{align*}
Since $\omega>\beta_3>\frac{12\underline{\lambda}_L^{-1}+1}{3}$, we have $\epsilon_4>0$.
From $\omega>\beta_3\geq\sqrt{\beta_6}$, one obtains that
\begin{align*}
    \epsilon_6=\frac{1}{8}-(\frac{2(1+\beta_1)}{\omega^2\underline{\lambda}_L}+\frac{1}{\omega^2\underline{\lambda}_L^2})L_f^2> 0.\addtag\label{eqn:bound2}
\end{align*}
From $\eta=\frac{\beta_2}{\omega}$ and $\omega> \beta_3\geq\frac{\beta_2}{\beta_4}$, we have $\eta<\beta_4$. Then it  holds that
$\eta\tilde{\epsilon}_1-\eta^2\tilde{\epsilon}_2$, $\eta(\epsilon_4-\eta\epsilon_5)$, $\eta(\epsilon_6-\eta\epsilon_7)$, and $\epsilon_{10}-\eta\epsilon_{11}-\eta^2\epsilon_{12}$ are positive.
From $\eta=\frac{\beta_2}{\omega}$ and $\omega>\beta_3$, we have
\begin{align*}
    \epsilon_8&=(\frac{2(1+\beta_1)^2}{\beta_2\omega\underline{\lambda}_L}+\frac{1+\beta_1}{\omega^2\underline{\lambda}_L}\\
	&~~~+\frac{1}{\beta_2\omega\underline{\lambda}_L^2}+\frac{1}{\omega^2\underline{\lambda}_L^2}+\frac{3}{2})L_f^2+L_f< c_1.\addtag\label{eq:upperofep8}
\end{align*}

\textbf{(ii)} From~\eqref{eq:lyapunovofv4} and~\eqref{eq:upperofV11}, we have
\begin{align*}
    \E[V_{4,k+1}]&\leq \E[V_{4,k}]-\frac{\eta}{4}\E\|\bg_k\|^2+\frac{\eta L_f^2}{2}\E\|\x_k\|_\bk^2-\frac{\eta}{4}\E\|\bg_k^b\|^2\\
	&~~~+\eta^2L_f(\sigma^2+\E\|\bg_k\|^2)\\
    &\leq \E[V_{4,k}]+\frac{\eta L_f^2}{2}\E\|\x_k\|_\bk^2-\frac{\eta}{4}\E\|\bg_k^b\|^2+\eta^2L_f\sigma^2,\addtag\label{eq:upperofv4}
\end{align*}
where the last inequality holds due to $\eta=\frac{\beta_2}{\omega}$ and $\omega>\beta_3\geq 4\beta_2 L_f$.

\textbf{(iii)} We denote the following useful function
\begin{align*}
U_k&=\Vert\x_{k}\Vert_\bk^2+\left\Vert\vv_{k}+\frac{1}{\omega}\g_{k}^b\right\Vert^2_{\bp}+\Vert\x_{k}-\x^c_{k}\Vert^2\\
&~~~+n(f(\bar{x}_{k})-f^*).
\end{align*}
From $\gamma=\beta_1\omega$, we have
\begin{align*}
    V_k&=\frac{1}{2}\Vert\x_{k}\Vert_\bk^2+\frac{1}{2}\left\Vert\vv_{k}+\frac{1}{\omega}\g_{k}^b\right\Vert^2_{\bp+\beta_1\bp}\\
	&~~~+\x_{k}^\top\bk\bp(\vv_{k}+\frac{1}{\omega}\g_{k}^b)\\
    &~~~+n(f(\bar{x}_{k})-f^*)+\Vert\x_{k}-\x^c_{k}\Vert^2\\
    &\geq\frac{1}{2}\Vert\x_{k}\Vert_\bk^2+\frac{1}{2}(1+\beta_1)\left\Vert\vv_{k}+\frac{1}{\omega}\g_{k}^b\right\Vert^2_{\bp}\\
	&~~~-\frac{\omega}{2\gamma\underline{\lambda}_L}\|\x_k\|_\bk^2-\frac{\gamma}{2\omega}\left\Vert\vv_{k}+\frac{1}{\omega}\g_{k}^b\right\Vert^2_{\bp}\\
    &~~~+n(f(\bar{x}_{k})-f^*)+\Vert\x_{k}-\x^c_{k}\Vert^2\\
    &\geq \check{c}_1U_k\geq 0,\addtag\label{eq:lowerofw}
\end{align*}
where the last inequality holds due to $\check{c}_1=\frac{1}{2}-\frac{1}{2\beta_1\underline{\lambda}_L}>\frac{1}{2}-\frac{1}{2c_0\underline{\lambda}_L}>\frac{1}{2}-\frac{1}{9}=\frac{7}{18}$. From $\eta=\frac{\beta_2}{\omega}$,~\eqref{eq:upperofV14}, and~\eqref{eq:upperofep8}, it holds that
\begin{align*}
    \E[V_{k+1}]&\leq \E[V_k]-c_2\E\|\x_k\|_\bk^2-\frac{\beta_2}{4\omega}\E\|\bg_k^b\|^2\\
    &~~~+\frac{(c_1+n\epsilon_9)\beta_2^2\sigma^2}{\omega^2}+\epsilon_{13}h_k^2\sigma_{\mathcal{C}}.\addtag\label{eq:upperofx1}
    \end{align*}
Then summing~\eqref{eq:upperofx1} over $k\in[0,T]$, we have 
\begin{align*}
\E[V_{k+1}]+&\sum_{k=0}^T\E[c_2\|\x_k\|_\bk^2+\frac{\beta_2}{4\omega}\|\bg_k^b\|^2]\\
	\leq& V_0+\frac{(T+1)(c_1+n\epsilon_9)\beta_2^2\sigma^2}{\omega^2}+\sum_{k=0}^T\epsilon_{13}h_k^2\sigma_{\mathcal{C}}.\addtag\label{eq:upperofx2}
\end{align*}
From $c_2>0$, $\beta_2>0$, $\omega>0$, $h_k=h_0^k$, $0<h_0<1$,~\eqref{eq:lowerofw}, and~\eqref{eq:upperofx2}, we have
\begin{align*}
    \frac{1}{T+1}&\sum_{k=0}^T\mathbb{E}\left[\frac{1}{n}\sum_{i=1}^n\Vert x_{i,k}-\bar{x}_k\Vert^2\right]  \\
    \leq&\frac{V_0}{nc_2(T+1)}+\frac{(c_1+n\epsilon_9)\beta_2^2\sigma^2}{nc_2\omega^2}+\frac{\epsilon_{13}h_0\sigma_{\mathcal{C}}}{nc_2(T+1)(1-h_0)}.\addtag\label{eq:upperofx3}
\end{align*}
Since Assumption~\ref{as:finite}, it holds that $V_0=\mathcal{O}(n)$. Then we have~\eqref{eq:theo11}. From summing~\eqref{eq:upperofv4} over $k\in[0,T]$, one obtains that
\begin{align*}
    \frac{1}{4}\sum_{k=0}^T\E[n\|\lf(\bar{x}_k)\|^2]&=\frac{1}{4}\sum_{k=0}^T\E[\|\bg_k^b\|^2]\\
   & \leq \frac{V_{4,0}}{\eta}+\frac{L_f^2}{2}\sum_{k=0}^T\E[\|\x_k\|^2_\bk]\\
	&~~~+(T+1)L_f^2\sigma^2\eta.\addtag\label{eq:upperofx4}
\end{align*}
From~\eqref{eq:upperofx3},~\eqref{eq:upperofx4} and $\eta=\frac{\beta_2}{\omega}$, we have
\begin{align*}
    \frac{1}{T+1}\sum_{k=0}^{T}\E[\|\lf(\bar{x}_k)\|^2]
   & \leq \frac{4\omega(f(\bar{x}_0)-f^*)}{\beta_2(T+1)}+\frac{4L_f^2\sigma^2\beta_2}{n\omega}\\
	&~~~+\mathcal{O}(\frac{1}{T+1})+\mathcal{O}(\frac{1}{\omega^2}).\addtag\label{eq:upperofx5}
\end{align*}
Then we complete the proof.
\section{The proof of Theorem~\ref{theo:convergence2}}\label{app-convergence2}
In this proof, in addition to the notations in Appendix~\ref{app-convergence1}, we also denote
\begin{align*}
    &c_3=\eta(\epsilon_4-\eta\epsilon_5)\\
    &\beta_8 = \max\{\frac{1}{2}+\beta_1,\frac{\gamma \underline{\lambda}_L+\omega}{2 \gamma \underline{\lambda}_L}\}\\
    &\beta_9=\frac{1}{\beta_8}\min\{c_2,c_3,\frac{\nu}{2(T+1)^\theta}\}\\
     &\beta_{10}=\frac{1}{\beta_8}\min\{c_2,c_3\}
\end{align*}
From the conditions in Theorem~\ref{theo:convergence2}, we know that the Lemma~\ref{lemma:fixedinquali} still holds.
From the Assumptions~\ref{as:finite} and~\ref{as:PLcondition}, one obtains that
\begin{align*}
    \|\bg_k^b\|^2=n\|\lf(\bar{x}_k)\|^2\geq 2n\nu(f(\bar{x}_k)-f^*)=2\nu V_{4,k}.\addtag\label{upperofV21}
\end{align*}
From~\eqref{eq:lowerofw}, we have
\begin{align*}
\Vert\x_{k}\Vert_\bk^2+n(f(\bar{x}_{k})-f^*)\leq U_k\leq \frac{V_k}{\check{c}_1}.\addtag\label{eq:proveoftheo21}
\end{align*}
Similar to~\eqref{eq:lowerofw}, we have
\begin{align*}
   V_k\leq \beta_8 U_k.\addtag\label{upperofV2}
\end{align*}
Due to $T\geq(\beta_3/\beta_2)^{1/\theta}$, we have $\omega=\beta_2(T+1)^\theta>\beta_3\geq\frac{\beta_2}{\beta_4}$. Combining $\eta=\frac{\beta_2}{\omega}$, we have $\eta<\beta_4$. From~$\eta<\beta_4$,~\eqref{eq:upperofV14},~\eqref{upperofV21}, and~\eqref{upperofV2}, it holds that
\begin{align*}
    \E[V_{k+1}]&\leq \E[V_k]-\E\|\x_k\|_{(\eta\tilde{\epsilon}_1-\eta^2\tilde{\epsilon}_2)\bk}\\
        &~~~-\E\left\Vert\vv_k+\frac{1}{\omega}\g_k^b\right\Vert^2_{\eta(\epsilon_4-\eta\epsilon_5)\bp}-\frac{\eta\nu}{2}\E[V_{4,k}]\\
&~~~+\epsilon_8\sigma^2\eta^2+\epsilon_9n\sigma^2\eta^2+\epsilon_{13}h_k^2\sigma_{\mathcal{C}}\\
&\leq \E[V_k]-\frac{1}{\beta_8}\min\{c_2,c_3,\frac{\eta\nu}{2}\}\E[V_k]\\
        &~~~+\epsilon_8\sigma^2\eta^2+\epsilon_9n\sigma^2\eta^2+\epsilon_{13}h_k^2\sigma_{\mathcal{C}}\\
        &\leq \E[V_k]-\beta_9\E[V_k]+\frac{(\epsilon_8+\epsilon_9n)\sigma^2}{(T+1)^{2\theta}}+\epsilon_{13}h_k^2\sigma_{\mathcal{C}}.\addtag\label{eq:proveoftheo22}
\end{align*}
Since $\beta_1>c_0>1$, then we have $\beta_8>1$. Combining $\beta_2<1$, it holds that
\begin{align*}
   \beta_9\leq c_3 <\eta\epsilon_4&= \eta\left(\frac{3\omega-1}{4}-3\underline{\lambda}_L^{-1}\right)<\frac{3\beta_2}{4}<\frac{3}{4}.\addtag\label{eq:proveoftheo23}
\end{align*}
 Then, from~\eqref{eq:proveoftheo23} and $\eta<\beta_4$, one obtains that 
\begin{align*}
    0<\beta_9<\frac{3}{4}.\addtag\label{eq:proveoftheo24}
\end{align*}
From~\eqref{eq:lowerofw},~\eqref{eq:proveoftheo22}, and~\eqref{eq:proveoftheo24}, we have
\begin{align*}
   \E[V_{k+1}]&\leq (1-\beta_9)^{k+1}\E[V_0]+\frac{(\epsilon_8+\epsilon_9n)\sigma^2}{(T+1)^{2\theta}}\sum_{l=0}^k(1-\beta_9)^l\\
   &~~~+\sum_{l=0}^k\epsilon_{13}\sigma_{\mathcal{C}}(1-\beta_9)^l h_0^{2(k-l)}\\
   &\leq(1-\beta_9)^{k+1}\E[V_0]+\frac{(\epsilon_8+\epsilon_9n)\sigma^2}{\beta_9(T+1)^{2\theta}}\\
   &~~~+\epsilon_{13}\sigma_{\mathcal{C}}(1-\beta_9-h_0^2)^{k+1}\\
   &\leq(1-\beta_9)^{k+1}(\E[V_0]+\epsilon_{13}\sigma_{\mathcal{C}})+\frac{(\epsilon_8+\epsilon_9n)\sigma^2}{\beta_9(T+1)^{2\theta}},\addtag\label{eq:proveoftheo25}
\end{align*}
where the last inequality holds due to $h_0^2<\frac{1}{4}<1-\beta_9$ and~\cite[Lemma~5]{yi2021linear}. Since $\beta_9=\mathcal{O}(1/(T+1)^\theta)$, $\epsilon_{13}=\mathcal{O}(1)$, $\theta\in(0,1)$,~\eqref{eq:linearbound},~\eqref{eq:proveoftheo21}, and~\eqref{eq:proveoftheo25}, we have
\begin{align*}
    \Vert\x_{k}\Vert_\bk^2+n(f(\bar{x}_{k})-f^*)=\mathcal{O}(\frac{n}{T^\theta}),\forall k\leq T,\addtag\label{eq:proveoftheo26}
\end{align*}
which means that there exists a constant $c_4>0$ such that
\begin{align*}
    \Vert\x_{k}\Vert_\bk^2+n(f(\bar{x}_{k})-f^*)<nc_4,\addtag\label{eq:proveoftheo27}
\end{align*}
Denote $\breve{V}_{k+1}=V_{k+1}-V_{4,k+1}$, combining~\eqref{eq:lyapunovofv4} and~\eqref{eq:upperofV14}, we have
\begin{align*}
      \E[\breve{V}_{k+1}]&\leq \E[\breve{V}_k]-\E\|\x_k\|_{(\eta\tilde{\epsilon}_1-\eta^2\tilde{\epsilon}_2)\bk}\\
        &~~~-\E\left\Vert\vv_k+\frac{1}{\omega}\g_k^b\right\Vert^2_{\eta(\epsilon_4-\eta\epsilon_5)\bp}+\epsilon_7\eta^2\E\|\bg_k\|^2\\
 &~~~+\epsilon_8\sigma^2\eta^2+\epsilon_9n\sigma^2\eta^2+\epsilon_{13}h_k^2\sigma_{\mathcal{C}},\addtag\label{eq:proveoftheo28}
\end{align*}
From~\eqref{eq:propertyofbg2} and Cauchy--Schwarz inequality, we have 
\begin{align*}
    \E_{\xi_k}[\Vert\bg_k\Vert^2]&\leq 2\E_{\xi_k}[\Vert\bg_k-\bg^b_k\Vert^2]+2\Vert\g^b_k\Vert^2\\
    &\leq 2L_f^2\Vert\x_k\Vert^2_\bk+2\Vert\g^b_k\Vert^2.\addtag\label{eq:proveoftheo29}
\end{align*}
From~\eqref{eq:lsmooth}, we have
\begin{align*}    \Vert\g^b_k\Vert^2=n\|\lf(\bar{x}_k)\|^2\leq 2nL_f(f(\bar{x}_k)-f^*)=2L_f V_{4,k}\addtag\label{eq:proveoftheo30}
\end{align*}
Since $\beta_1>1$, we have $\beta_8>\frac{1}{2}+1>\frac{3}{2}$. Combining~\eqref{eq:lowerofw},~\eqref{upperofV2}, one obtains that
\begin{align*}    
0\leq\breve{V}_{k}\leq\beta_8 U_k-V_{4,k}<\beta_8(V_{1,k}+V_{2,k})\addtag\label{eq:proveoftheo31}
\end{align*}
Similar to~\eqref{eq:proveoftheo24}, it holds that $0<\beta_{10}<\frac{3}{4}$. From~\eqref{eq:proveoftheo28}--\eqref{eq:proveoftheo31}, we have
\begin{align*}    
\E[\breve{V}_{k+1}]&\leq (1-\beta_{10})\E[\breve{V}_k]+\epsilon_{13}h_k^2\sigma_{\mathcal{C}}\\
&~~~+\eta^2(2L_f^2\epsilon_7nc_4+4L_f\epsilon_7nc_4+\epsilon_8\sigma^2+\epsilon_9n\sigma^2)\\
&\leq(1-\beta_{10})^{k+1}\E[\breve{V}_0]+\epsilon_{13}\sigma_{\mathcal{C}}(1-\beta_{10}-h_0^2)^{k+1}\\
&~~~+\frac{\eta^2}{\beta_{10}}(2L_f^2\epsilon_7nc_4+4L_f\epsilon_7nc_4+\epsilon_8\sigma^2+\epsilon_9n\sigma^2).\addtag\label{eq:proveoftheo32}
\end{align*}
where the last inequality holds due to $h_0^2<\frac{1}{4}<1-\beta_{10}$. Since $\eta=1/(T+1)^\theta$, we have~\eqref{eq:theo21}. From~\eqref{eq:lyapunovofv4} and~\eqref{eq:upperofV11}, we have
\begin{align*}    
\E_{\xi_k}[V_{4,k+1}]&\leq V_{4,k}-\frac{\eta}{4}\|\bg_k\|^2+\frac{\eta L_f^2}{2}\|\x_k\|_\bk^2-\frac{\eta}{4}\|\bg_k^b\|^2\\
 &~~~+\eta^2L_f(\sigma^2+\|\bg_k\|^2),\addtag\label{eq:proveoftheo33}
\end{align*}
From $\eta=\frac{\beta_2}{\omega}$ and $\omega>\beta_3>4\beta_2L_f$, we know that $\eta L_f<\frac{1}{4}$. Then combining~\eqref{upperofV21},~\eqref{eq:proveoftheo33} can be rewritten as
\begin{align*}    
\E_{\xi_k}[V_{4,k+1}]&\leq V_{4,k}+\frac{\eta L_f^2}{2}\|\x_k\|_\bk^2-\frac{\eta}{4}\|\bg_k^b\|^2+\frac{\eta}{4}\sigma^2\\
&\leq(1-\frac{\eta\nu}{4}) V_{4,k}+\frac{\eta L_f^2}{2}\|\x_k\|_\bk^2+\frac{\eta}{4}\sigma^2\\
&\leq(1-\frac{\eta\nu}{4})^{k+1} V_{4,0}+\frac{1}{\nu}( L_f^2\|\x_k\|_\bk^2+\frac{\sigma^2}{2}),\addtag\label{eq:proveoftheo34}
\end{align*}
From $\eta=1/(T+1)^\theta$,~\eqref{eq:theo21},~\eqref{eq:linearbound}, and~\eqref{eq:proveoftheo34}, we have~\eqref{eq:theo22}

\section{The proof of Theorem~\ref{theo:convergence3}}\label{app-convergence3}
In this proof, in addition to the notations in Appendix~\ref{app-convergence1}, we also denote
\begin{align*}
&\breve{\epsilon}_{1,k}=\frac{\gamma_k}{2}\underline{\lambda}_L-(\frac{9\omega_k+4}{4}+\frac{5}{4}L_f^2)\\
&\breve{\epsilon}_{2,k}=(12+\frac{16}{\varphi_1})L_f^2+(4+\frac{8}{\varphi_1}\gamma_k^2\bar{\lambda}_L^2)+1+2\omega_k^2+3\gamma_k^2\bar{\lambda}_L^2\\
&\breve{\epsilon}_{3,k}=\frac{\gamma_k}{4}\bl-\frac{1}{2}(1+L_f^2)\bk\\
&\breve{\epsilon}_{4,k}=(\frac{5L_f^2}{2}+\frac{3}{2})\bk+(3\gamma_k^2+\omega_k(\omega_k+\gamma_k))\bl\\
&\breve{\epsilon}_{5,k}=\eta_k(\frac{3\omega_k-1}{4}-3\underline{\lambda}_L^{-1})\\
&~~~~~-\eta_k^2(\omega_k^2\bar{\lambda}_L+\underline{\lambda}_L^{-1}+(4+\frac{8}{\varphi_1})\omega_k^2\underline{\lambda}_L)\\
&~~~~~~-b_k(\frac{3\eta_k^2\omega_k^2\bar{\lambda}_L}{2}+\frac{\eta_k\omega_k}{4}+\frac{1}{2}+\frac{\beta_1}{4}-2\eta_k\gamma_k)\\
&\breve{\epsilon}_{6,k}=\frac{1}{8}-(\frac{2(1+\beta_1)}{\omega_k^2\underline{\lambda}_L}+\frac{1}{\omega_k^2\underline{\lambda}_L^2})L_f^2\\
&~~~~~~-4\eta_k L_f^2(\frac{1}{2\underline{\lambda}_L}(b_k+b_k\beta_1+b_k^2+b_k^2\beta_1)+\frac{b_k}{2})\\
&~~~~~~-\eta_k((\frac{1+\beta_1}{\omega_k^2\underline{\lambda}_L}+\frac{1}{\omega_k^2\underline{\lambda}_L^2}+\frac{3}{2})L_f^2+L_f)\\
&~~~~~~-\eta_kL_f^2b_k(\beta_1^2+2\left(\frac{(1+\beta_1)^2}{\eta_k\omega_k^2}+\frac{1+\beta_1}{2\omega_k^2}\right)\frac{1}{\underline{\lambda}_L}+1)\\
&\breve{\epsilon}_{7,k}=(\frac{2(1+\beta_1)^2}{\eta\omega_k^2\underline{\lambda}_L}+\frac{1+\beta_1}{\omega_k^2\underline{\lambda}_L}+\frac{1+\eta}{\eta\omega_k^2\underline{\lambda}_L^2}+\frac{3}{2})L_f^2+L_f\\
&~~~~~~+4 L_f^2(\frac{1}{2\underline{\lambda}_L}(b_k+b_k\beta_1+b_k^2+b_k^2\beta_1)+\frac{b_k}{2})\\
&~~~~~+L_f^2b_k(\beta_1^2+2\left(\frac{(1+\beta_1)^2}{\eta_k\omega_k^2}+\frac{1+\beta_1}{2\omega_k^2}\right)\frac{1}{\underline{\lambda}_L}+1)\\
&\breve{\epsilon}_{8,k}=11+\frac{16}{\varphi_1}+2b_k\\
&\breve{\epsilon}_{9,k}=\frac{\varphi_1}{2}+\frac{\varphi_1^2}{2}\\
&\breve{\epsilon}_{10,k}=\frac{1+b_k}{2}(\gamma_k+2 \omega_k)\bar{\lambda}_Lr_0+2\omega_k r_0+\frac{\gamma_k}{4}\bar{\lambda}_Lr_0\\
&\breve{\epsilon}_{11,k}=\frac{(8+(7+6b_k)\varphi_1)\gamma_k^2\bar{\lambda}_L^2r_0}{\varphi_1}\\
	&~~~~~~~~~+(1+b_k+2\omega_k^2+2b_k\omega_k(\omega_k+\gamma_k)\bar{\lambda}_L)r_0\\
&\breve{\epsilon}_{12,k}=\frac{1}{\underline{\lambda}_L}(b_k+b_k\beta_1+b_k^2+b_k^2\beta_1)+b_k\\
&\breve{\epsilon}_{13,k}=\eta_k(\gamma_k+6 \omega_k)\bar{\lambda}_Lr^2+\eta_k^2(14\gamma_k^2\bar{\lambda}_L^2r^2+2+4\omega^2)+1\\
&~~~~~~~~~+b_kr^2((\frac{5\omega_k}{2}+\gamma_k)\bar{\lambda}_L\eta_k\\
&~~~~~~~~~+(\frac{(\omega_k^2+\omega_k\gamma_k)\bar{\lambda}_L}{2}+\frac{3\gamma_k^2\bar{\lambda}_L^2}{4}+1)\eta_k^2)\\
   &\bar{\sigma}=2L_f f^*-2L_f\frac{1}{n}\sum_{i=1}{n}f_i^*\\
\end{align*}

\begin{lemma}
    Suppose Assumptions~\ref{as:strongconnected}--\ref{as:boundedvar} and~\ref{as:compressor} hold. If $\gamma_k=\beta_1\omega$, $\beta_1>1$, $\alpha_x\in(0,\frac{1}{r})$, and $\eta_k=\beta_2/\omega_k$, it  holds that
    \begin{align*}
     \E_{\xi_k}[V_{k+1}]&\leq V_k-\|\x_k\|_{(\eta_k\breve{\epsilon}_{1,k}-\eta_k^2\breve{\epsilon}_{2,k})\bk+b_k\eta_k(\breve{\epsilon}_{3,k}-\eta_k\breve{\epsilon}_{4,k})}\\
	&~~~-\Vert\vv_k+\frac{1}{\omega_k}\g_k^b\Vert^2_{\breve{\epsilon}_{5,k}\bp}\\
 &~~~-\eta_k\breve{\epsilon}_{6,k}\|\bg_k\|-\frac{\eta}{4}\|\bg_k^b\|^2+\breve{\epsilon}_{7,k}\sigma^2\eta^2+\breve{\epsilon}_{8,k}n\sigma^2\eta^2\\
 &~~~-(\breve{\epsilon}_{9,k}-\eta_k\breve{\epsilon}_{10,k}-\eta^2_k\breve{\epsilon}_{11,k})\Vert\x_k-\x_{k}^c\Vert^2\\
	&~~~+\eta_k^2\breve{\epsilon}_{12,k}(2L_fV_{4,k}+n\bar{\sigma})+\breve{\epsilon}_{13,k}h_k^2\sigma_{\mathcal{C}},\addtag\label{eq:upperofV24}
    \end{align*}
    where $V_{k+1}=\sum_{i=1}^5 V_{i,k+1}$.
\end{lemma}
\begin{proof}
\textbf{(i)} We have
\begin{align*}
\|\g_k^b\|^2&=\sum_{i=1}^n\|\lf_i(\bar{x}_k)\|^2\leq\sum_{i=1}^n2L_f(f(x)-f^*)\\
    &=2nL_f(f(x)-f^*)+n\bar{\sigma},\addtag\label{eq:upperofV21}
\end{align*}
the inequality cones from~Assumption~\ref{as:strongconnected} and~\eqref{eq:lsmooth}. Furthermore, we also have
\begin{align*}
    \|\g_{k+1}^b\|^2&=\|\g_{k+1}^b-\g_{k}^b+\g_{k}^b\|^2\leq2\|\g_{k+1}^b-\g_{k}^b\|^2+2\|\g_{k}^b\|^2\\
    &\leq2(\eta_k^2L_f^2\Vert\bg_k^s\Vert^2+2nL_f(f(x)-f^*)+n\bar{\sigma}),\addtag\label{eq:upperofV22}
\end{align*}
the first and second inequalities hold due to~Cauchy–Schwarz inequality and~\eqref{eq:propertyofbg3}, respectively. 

From Lemma~\ref{lemma:lyapunov},~\eqref{eq:propertyofcompressors1},~\eqref{eq:propertyofk1},~\eqref{eq:upperofV11},~\eqref{eq:upperofV12} and~\eqref{eq:upperofV22}, we have~\eqref{eq:upperofV24}
\end{proof}

We introduce the following useful lemmas.
\begin{lemma}\label{lemma:useful}
    \cite{yi2022primal} Let $\{z_k\},~\{r_{1,k}\},$ and $\{r_{2,k}\}$ be sequences. Suppose there exists $t_1\in\mathbb{N}_+$ such that
    \begin{align*}
        &z_k\geq 0,~z_{k+1}\leq(1-r_{1,k})z_k+r_{2,k},\\
        & 1>r_{1, k} \geq \frac{a_1}{\left(k+t_1\right)^\delta} \\
& r_{2, k} \leq \frac{a_2}{\left(k+t_1\right)^2}, \quad \forall k \in \mathbb{N}
    \end{align*}
    where $\delta>0$, $a_1>0$ and $a_2>0$ are constants.

    (i) if $\delta=1$, then 
    \begin{align*}
        z_k \leq \phi_1\left(k, t_1, a_1, a_2, z_0\right), \forall k \in \mathbb{N}_+,\addtag\label{eq:lemma61}
    \end{align*}
    where
    \begin{align*}
        \phi_1\left(k, t_1, a_1, a_2, z_0\right)= & \frac{t_1^{a_1} z_0}{\left(k+t_1\right)^{a_1}}+\frac{a_2}{\left(k+t_1-1\right)^2} \\
& +4 a_2 s_1\left(k+t_1\right),\addtag\label{eq:lemma62}
    \end{align*}
    with
    \begin{align*}
        s_1(k)= \begin{cases}\frac{1}{\left(a_1-1\right) k}, & \text { if } a_1>1, \\ \frac{\ln (k-1)}{k}, & \text { if } a_1=1, \\ \frac{-t_1^{a_1-1}}{\left(a_1-1\right) k^{a_1}}, & \text { if } a_1<1 .\end{cases}
    \end{align*}

    (ii) if $\delta=0$, then 
    \begin{align*}
        z_k \leq \phi_2\left(k, t_1, a_1, a_2, z_0\right), \forall k \in \mathbb{N}_{+}\addtag\label{eq:lemma63}
    \end{align*}
    where
    \begin{align*}
        \phi_2\left(k, t_1, a_1, a_2, z_0\right)= & \left(1-a_1\right)^k z_0\\
	&+a_2\left(1-a_1\right)^{k+t_1-1}\left(\left[t_2-t_1\right]s_2\left(t_1\right)\right. \\
& \left.+\left(\left[t_3-t_1\right]-\left[t_2-t_1\right]\right) s_2\left(t_3\right)\right) \\
& +\frac{\mathbf{1}_{\left(k+t_1-1 \geq t_3\right)} 2 a_2}{-\ln \left(1-a_1\right)\left(k+t_1\right)^2\left(1-a_1\right)},\addtag\label{eq:lemma64}
    \end{align*}
    with $s_2(k)=\frac{1}{k^2\left(1-a_1\right)^k}, t_2=\left\lceil\frac{-2}{\ln \left(1-a_1\right)}\right\rceil \text {, and } t_3=\left\lceil\frac{-4}{\ln \left(1-a_1\right)}\right\rceil \text {. }$
\end{lemma}

We then ready to prove Theorem~\ref{theo:convergence3}. We also denote
\begin{align*}
    &m_1=(\frac{\beta_1\underline{\lambda}_L}{2}-\frac{9}{4})-1\\
    &m_2=(12+\frac{16}{\varphi_1})+(4+\frac{8}{\varphi_1}\beta_1^2\bar{\lambda}_L^2)+1+2+3\beta_1^2\bar{\lambda}_L^2\\
    &m_3=\frac{\beta_1\underline{\lambda}_L}{4}-1\\
    &m_4=4+(3\beta_1^2+1+\beta_1)\bar{\lambda}_L\\
    &m_5=\frac{\beta_2}{4}-\beta_2^2(\bar{\lambda}_L+\underline{\lambda}_L^{-1}+(4+\frac{8}{\varphi_1})\underline{\lambda}_L)\\
    &m_6=4\eta_k L_f^2(\frac{1}{2\underline{\lambda}_L}(b_k+b_k\beta_1+b_k^2+b_k^2\beta_1)+\frac{b_k}{2})\\
&~~~~~~+\eta_kL_f^2b_k(\beta_1^2+2\left(\frac{(1+\beta_1)^2}{\eta_k\omega_k^2}+\frac{1+\beta_1}{2\omega_k^2}\right)\frac{1}{\underline{\lambda}_L}+1)\\
    &m_7=\frac{\varphi_1}{2}+\frac{\varphi_1^2}{2}\\
&m_8=(\beta_1+2)\bar{\lambda}_Lr_0+2 r_0+\frac{\beta_1}{4}\bar{\lambda}_Lr_0\\
&m_9=\frac{21\varphi_1)\beta_1^2\bar{\lambda}_L^2r_0}{\varphi_1}+(4+2(1+\beta_1)\bar{\lambda}_L)r_0\\
&m_{10}=\frac{3\beta_2^2\bar{\lambda}_L}{2}+\frac{\beta_2}{4}+\frac{1}{2}+\frac{\beta_1}{4}-2\beta_1\beta_2\\
&m_{11}=\beta_2 L_f^2\bigg(\frac{2}{\underline{\lambda}_L}(2+2\beta_1)+\beta_1^2\\
&~~~~~~+2\left(\frac{(1+\beta_1)^2}{\beta_2}+\frac{1+\beta_1}{2}\right)\frac{1}{\underline{\lambda}_L}+3\bigg)\\
&m_{12}=\bigg(\frac{2(1+\beta_1)^2}{\beta_2\beta_0t_1\underline{\lambda}_L}+\frac{1+\beta_1}{\beta_0^2t_1^2\underline{\lambda}_L}+\frac{1}{\beta_2\beta_0t_1\underline{\lambda}_L^2}+\frac{1}{\beta_0^2t_1^2\underline{\lambda}_L^2}\\
&~~~~~~+\frac{3}{2}\bigg)L_f^2+L_f\\
&~~~~~~+4L_f^2\frac{1}{\beta_0t_1}(\frac{1}{\underline{\lambda}_L}(2+2\beta_1)+\frac{1}{2})\\
&~~~~~~+L_f^2\frac{1}{\beta_0t_1}(\beta_1^2+2\left(\frac{(1+\beta_1)^2}{\beta_2\beta_0t_1}+\frac{1+\beta_1}{2\beta_0^2t_1^2}\right)\frac{1}{\underline{\lambda}_L}+1)\\
&m_{13}=11+\frac{16}{\varphi_1}+2\\
&m_{14}=\frac{1}{\underline{\lambda}_L}(2+2\beta_1)+1\\
&m_{15}=\beta_2(6+\beta_1)\bar{\lambda}_Lr^2+14\beta_1^2\beta_2^2\bar{\lambda}_L^2r^2+\frac{2\beta_2^2}{\beta_0^2}+4\beta_2^2+1\\
&~~~~~~~~~+r^2((\frac{5\beta_2}{2}+\beta_1\beta_2)\bar{\lambda}_L\\
&~~~~~~~~~+(\frac{(\beta_2^2+\beta_1\beta_2^2)\bar{\lambda}_L}{2}+\frac{3\beta_1^2\beta_2^2\bar{\lambda}_L^2}{4}+\frac{\beta_2^2}{\beta_0^2}))\\
&m_{16}=(\frac{1+\beta_1}{\beta_0^2t_1^2\underline{\lambda}_L}+\frac{1}{\beta_0^2t_1^2\underline{\lambda}_L^2}+\frac{3}{2})L_f^2+L_f+\frac{m_{11}}{\beta_2\beta_0t_1}\\
&m_{17}=\frac{1}{\beta_2^2\underline{\lambda}_L}(2+2\beta_1)+\frac{1}{\beta_2^2}\\
&m_{18}=((\frac{m_{11}}{n}+m_{12})\sigma^2+m_{17}\bar{\sigma})\\
&m_{19}=\frac{1}{\beta_8}\min\{\frac{\bar{c}_4}{\eta_k},\frac{m_5}{2\eta_k},\frac{m_7}{2\eta_k},\frac{\nu}{4}\}\\
&\bar{c}_0=\max\{\frac{2m_{10}}{m_5},16m_{11}\}\\
    &\bar{c}_1=\max\{\frac{9}{2\underline{\lambda}_L}+1\}\\
    &\bar{c}_2=\min\bigg\{\frac{m_1}{m_2},\frac{m_3}{m_4},\frac{\sqrt{m_8^2+2m_7m_9}-m_8}{2m_9}\\
	&~~~~~~~~,\frac{\bar{\lambda}_L+\underline{\lambda}_L^{-1}+(4+\frac{8}{\varphi_1})\underline{\lambda}_L}{4}\\
    &~~~~~~~~,\frac{1}{(\frac{1+\beta_1}{32\underline{\lambda}_L}+\frac{1}{32\underline{\lambda}_L^2}+\frac{3}{64})L_f^2+32L_f}\bigg\}\\
    &\bar{c}_3=\max\bigg\{1+\frac{5}{4}L_f^2,\frac{4}{3\underline{\lambda}_L},4L_f\sqrt{\frac{4(1+\beta_1)}{\underline{\lambda}_L}+\frac{2}{\underline{\lambda}_L^2}},4\beta_2L_f\bigg\}\\
    &\bar{c}_4=m_1\beta_2-m_2\beta_2^2\\
    &\bar{c}_5=\max\{\frac{\bar{c}_0}{\beta_0},\frac{8L_f}{\nu \beta_2}(\frac{1}{\underline{\lambda}_L}(2+2\beta_1)+1),\frac{\bar{c}_3}{\beta_0},2\}
\end{align*}

\textbf{(i)} 
We first show the corresponding parameters are positive. From~$\gamma_k=\beta_1\omega_k$, $\beta_1>\frac{9}{2\underline{\lambda}_L}+1$, $\omega_k\geq\beta_0t_1\geq(1+\frac{5}{4}L_f^2)$, and $\eta_k=\frac{\beta_2}{\omega_k}$, one obtains that
\begin{align*}
    \eta_k\breve{\epsilon}_{1,k}\bk\geq m_1\beta_2\bk,~m_1>0.\addtag\label{eq:theo3w1}
\end{align*}
Similarly, since $\omega_k\geq(1+\frac{5}{4}L_f^2)>\max\{1,L_f\}$, we have
\begin{align*}
    \eta_k^2\breve{\epsilon}_{2,k}\bk\leq m_2\beta_2^2\bk.\addtag\label{eq:theo3w2}
\end{align*}
From~$\gamma_k=\beta_1\omega_k$, $\beta_1>\frac{9}{2\underline{\lambda}_L}+1$, $\omega_k\geq\beta_0t_1\geq(1+\frac{5}{4}L_f^2)$, and $\eta_k=\frac{\beta_2}{\omega_k}$, one obtains that
\begin{align*}
    \eta_k\breve{\epsilon}_{3,k}\geq m_3\beta_2\bk, m_3>0.\addtag\label{eq:theo3w3}
\end{align*}
Similarly, since $\omega_k\geq(1+\frac{5}{4}L_f^2)>\max\{1,L_f\}$, we have
\begin{align*}
    &\eta_k^2\breve{\epsilon}_{4,k}\bk\leq m_4\beta_2^2\bk,\addtag\label{eq:theo3w4}
\end{align*}
From $\gamma_k=\beta_1\omega_k$, $\omega_k\geq\frac{4}{3\underline{\lambda}_L}$ and $\omega_k\geq(1+\frac{5}{4}L_f^2)>\max\{1,L_f\}$, we have
\begin{align*}
    \breve{\epsilon}_{5,k}\geq\breve{\epsilon}_{5,k}^0,\addtag\label{eq:theo3w5}
\end{align*}
where $\breve{\epsilon}_{5,k}^0=m_5-b_km_{10}$.
Since $\omega_k\geq4L_f\sqrt{\frac{4(1+\beta_1)}{\underline{\lambda}_L}+\frac{2}{\underline{\lambda}_L^2}}$, $w_k>1$, and $\beta_2<\frac{1}{(\frac{1+\beta_1}{32\underline{\lambda}_L}+\frac{1}{32\underline{\lambda}_L^2}+\frac{3}{64})L_f^2+32L_f}$, we have
\begin{align*}
    \breve{\epsilon}_{6,k}\geq\breve{\epsilon}_{6,k}^0,\addtag\label{eq:theo3w6}
\end{align*}
where $\breve{\epsilon}_{6,k}^0=\frac{1}{16}-m_6$.
From~$\omega_k=\beta_0(k+t_1)$, we have
\begin{align*}
    b_k&=\frac{1}{\omega_k}-\frac{1}{\omega_{k+1}}=\frac{1}{\beta_0}(\frac{1}{k+t_1}-\frac{1}{k+t_1+1})\\
    &=\frac{1}{\beta_0(k+t_1)(k+t_1+1)}\leq\frac{\beta_0}{\omega_k^2}.\addtag\label{eq:theo3w7}
\end{align*}
Since~$\gamma_k=\beta_1\omega_k$,~$\omega_k>1$, $b_k=\frac{1}{\omega_k}-\frac{1}{\omega_{k+1}}<\frac{1}{\omega_k}<1$ and $\omega_k\geq(1+\frac{5}{4}L_f^2)>\max\{1,L_f\}$, we have
\begin{align*}
    \breve{\epsilon}_{9,k}-\eta_k\breve{\epsilon}_{10,k}-&\eta^2_k\breve{\epsilon}_{11,k}\geq m_7-\beta_2m_8-\beta_2^2m_9,~m_9>0.\addtag\label{eq:theo3w8}
\end{align*}
Since $\beta_2<\bar{c}_2=\max\bigg\{\frac{m_1}{m_2},\frac{m_3}{m_4},\frac{\sqrt{m_8^2+2m_7m_9}-m_8}{2m_9},$ $\frac{\bar{\lambda}_L+\underline{\lambda}_L^{-1}+(4+\frac{8}{\varphi_1})\underline{\lambda}_L}{4}\bigg\}$,~\eqref{eq:theo3w1}--\eqref{eq:theo3w4}, and~\eqref{eq:theo3w7}, we have
\begin{align*}
    &\bar{c}_4>0,\addtag\label{eq:theo3w9}\\
    &m_3\beta_2-m_4\beta_2^2>0,\addtag\label{eq:theo3w10}\\
    &m_7-\beta_2m_8-\beta_2^2m_9>\frac{m_7}{2}>0,\addtag\label{eq:theo3w11}\\
    &m_5>0\addtag\label{eq:theo3w12}.
\end{align*}
From~\eqref{eq:theo3w5},~\eqref{eq:theo3w7},~\eqref{eq:theo3w12}, and $\beta_0>\bar{c}_0>\frac{2m_{10}}{m_5t_1}$, we have
\begin{align*}
    \breve{\epsilon}_{5,k}^0\geq m_5-\frac{m_{10}}{\beta_0t_1^2}\geq m_5-\frac{m_{10}}{\beta_0t_1}\geq\frac{1}{2}m_5>0.\addtag\label{eq:theo3w13}
\end{align*}
Furthermore, since $\omega_k>1$, $b_k=\frac{1}{\omega_k}-\frac{1}{\omega_{k+1}}<\frac{1}{\omega_k}<1$, it holds that $b_k^2<b_k$. Then from~\eqref{eq:theo3w6},~\eqref{eq:theo3w7} and~$\beta_0>\bar{c}_0>\frac{2m_{10}}{m_5t_1}$, we have
\begin{align*}
     \breve{\epsilon}_{6,k}^0\geq\frac{1}{16}-\frac{m_{11}}{\beta_0t_1}>0.\addtag\label{eq:theo3w14}
\end{align*}
Since $t_1>\bar{c}_5>2$, then we have $\eta_k=\beta_2/\omega_k=\beta_2/\beta_0(k+t_1)<\beta_2/\beta_0$. From $\omega_k=\beta_0(k+t_1)>\beta_0t_1$ and $b_k<1$, we have
\begin{align*}
    &\breve{\epsilon}_{7,k}<m_{11},\addtag\label{eq:theo3w15}\\
    &\breve{\epsilon}_{8,k}<m_{12},\addtag\label{eq:theo3w16}\\
    &\breve{\epsilon}_{12,k}<m_{14}.\addtag\label{eq:theo3w17}\\
    &\breve{\epsilon}_{13,k}<m_{15}.\addtag\label{eq:theo3w18}
\end{align*}
From~\eqref{eq:upperofV24},~\eqref{eq:theo3w1}--\eqref{eq:theo3w6} and~~\eqref{eq:theo3w8}--\eqref{eq:theo3w18}, we have
\begin{align*}
     \E_{\xi_k}[V_{k+1}]&\leq V_k-\|\x_k\|_{\bar{c}_4\bk}-\frac{m_5}{2}\Vert\vv_k+\frac{1}{\omega_k}\g_k^b\Vert^2_{\bp}\\
     &~~~-\frac{m_7}{2}\|\x_k-\x_k^c\|^2\\
 &~~~-\frac{\eta_k}{4}\|\bg_k^b\|^2+(m_{11}+m_{12}n)\sigma^2\eta_k^2\\
 &~~~+\eta_k^2m_{14}(2L_fV_{4,k}+n\bar{\sigma})+m_{15}h_k^2\sigma_{\mathcal{C}},\addtag\label{eq:upperofV31}
    \end{align*}
    Similarly, we have
    \begin{align*}
     \E_{\xi_k}[\breve{V}_{k+1}]&\leq \breve{V}_k-\|\x_k\|_{\bar{c}_4\bk}-\frac{m_5}{2}\Vert\vv_k+\frac{1}{\omega_k}\g_k^b\Vert^2_{\bp}\\
      &~~~-\frac{m_7}{2}\|\x_k-\x_k^c\|^2\\
 &~~~+m_{16}\eta_k^2\|\bg_k\|^2+(m_{11}+m_{12}n)\sigma^2\eta_k^2\\
 &~~~+\eta_k^2m_{14}(2L_fV_{4,k}+n\bar{\sigma})+m_{15}h_k^2\sigma_{\mathcal{C}}.\addtag\label{eq:upperofV32}
    \end{align*}
    
With respect to $\|\bg_k^b\|^2$, we have
\begin{align*}
    \|\bg_k\|^2&=\|\bg_k-\bg_k^b+\bg_k^b\|^2\\
    &\leq 2\|\bg_k-\bg_k^b\|^2+2\|\bg_k^b\|^2\leq2L_f^2\|\x_k\|_\bk^2+2\|\bg_k^b\|^2,\addtag\label{eq:upperofV33}
\end{align*}
the first inequality comes from~Cauchy–Schwarz inequality and the first inequality comes from~\eqref{eq:propertyofbg2}. From~\eqref{eq:lsmooth} and Cauchy–Schwarz inequality, we have
\begin{align*}
\|\bg_k^b\|^2&=n\|\lf_i(\bar{x}_k)\|^2\leq2nL_f(f(x)-f^*)\\
    &=2nL_f(f(x)-f^*)=2L_f V_{4,k},\addtag\label{eq:upperofV34}
\end{align*}
Then combining~\eqref{eq:upperofV32}--\eqref{eq:upperofV34}, we have
 \begin{align*}
     \E_{\xi_k}[\breve{V}_{k+1}]&\leq \breve{V}_k-\|\x_k\|_{\bar{c}_4\bk}-\frac{m_5}{2}\Vert\vv_k+\frac{1}{\omega_k}\g_k^b\Vert^2_{\bp}\\
      &~~~-\frac{m_7}{2}\|\x_k-\x_k^c\|^2\\
 &~~~+2m_{16}\eta_k^2L_f^2\|\x_k\|_\bk^2+(m_{11}+m_{12}n)\sigma^2\eta^2\\
 &~~~+2L_f\eta_k^2(2m_{16}+m_{17})V_{4,k}+\eta_k^2m_{17}n\bar{\sigma},\addtag\label{eq:upperofV35}
    \end{align*}
From~\eqref{eq:lyapunovofv4} and~\eqref{eq:upperofV11}, we have
\begin{align*}
    \E_{\xi_k}[V_{4,k+1}]&\leq V_{4,k}-\frac{\eta_k}{4}\|\bg_k\|^2+\frac{\eta_kL_f^2}{2}\|\x_k\|_\bk^2-\frac{\eta_k}{4}\|\bg_k^b\|^2\\
&~~~+\eta_k^2L_f(\sigma^2+\|\bg_k\|^2)\\
    &\leq V_{4,k}-\frac{\eta_k}{4}\|\bg_k^b\|^2+\frac{\eta_kL_f^2}{2}\|\x_k\|_\bk^2+\eta_k^2L_f\sigma^2,\addtag\label{eq:upperofV36}
\end{align*}
the last inequality holds due to $\eta_k=\frac{\beta_2}{\omega_k}$ and $\omega_k>4\beta_2L_f$.

\textbf{(ii)} From~\eqref{upperofV21} and~\eqref{eq:upperofV31}, we have
\begin{align*}
     \E_{\xi_k}[V_{k+1}]&\leq V_k-\|\x_k\|_{\bar{c}_4\bk}-\frac{m_5}{2}\Vert\vv_k+\frac{1}{\omega_k}\g_k^b\Vert^2_{\bp}\\
      &~~~-\frac{m_7}{2}\|\x_k-\x_k^c\|^2\\
 &~~~-\frac{\eta_k\nu}{2}V_{4,k}+(m_{11}+m_{12}n)\sigma^2\eta_k^2\\
 &~~~+\eta_k^2m_{14}(2L_fV_{4,k}+n\bar{\sigma})+m_{15}h_k^2\sigma_{\mathcal{C}}\\
 &\leq V_k-\|\x_k\|_{\bar{c}_4\bk}-\frac{m_5}{2}\Vert\vv_k+\frac{1}{\omega_k}\g_k^b\Vert^2_{\bp}+m_{18}n\eta_k^2\\
  &~~~-\frac{m_7}{2}\|\x_k-\x_k^c\|^2\\
 &~~~-2(\frac{1}{4}-\frac{1}{\nu}L_fm_{14}\eta_k)\nu\eta_kV_{4,k}+m_{15}h_k^2\sigma_{\mathcal{C}}.\addtag\label{eq:upperofV41}
\end{align*}
From~\eqref{eq:theo3w12}, $\omega_k>\beta_0t_1$, and $t_1>\bar{c}_5>\frac{8L_f}{\nu \beta_2}(\frac{1}{\underline{\lambda}_L}(2+2\beta_1)+1)$, we have
\begin{align*}
        \frac{1}{4}-&\frac{1}{\nu}L_fm_{14}\eta_k=\frac{1}{4}-\frac{1}{\nu}L_f(\frac{\beta_0}{\beta_2\omega_k\underline{\lambda}_L}(2+2\beta_1)+\frac{\beta_0}{\beta_2\omega_k})\\
        &\geq\frac{1}{4}-\frac{1}{\nu}L_f(\frac{1}{\beta_2t_1\underline{\lambda}_L}(2+2\beta_1)+\frac{1}{\beta_2t_1})\geq \frac{1}{8}.\addtag\label{eq:upperofV42}
\end{align*}
From~\eqref{eq:lowerofw},~\eqref{upperofV2}, and~\eqref{eq:upperofV41}, it holds that
\begin{align*}
     \E_{\xi_k}[V_{k+1}]&\leq V_k-\frac{\eta_k}{\beta_8}\min\{\frac{\bar{c}_4}{\eta_k},\frac{m_5}{2\eta_k},\frac{m_7}{2\eta_k},\frac{\nu}{4}\}V_k+m_{18}n\eta_k^2\\
     &~~~+m_{15}h_k^2\sigma_{\mathcal{C}}\\
     &\leq V_k-m_{19}\eta_kV_k+m_{18}n\eta_k^2+m_{15}h_k^2\sigma_{\mathcal{C}}.\addtag\label{eq:upperofV43}
\end{align*}
Denote $z_k=\E[V_{k}], r_{1,k}=m_{19}\eta_k$ and $r_{2,k}=m_{18}n\eta_k^2+m_{15}h_k^2\sigma_{\mathcal{C}}$, from~\eqref{eq:upperofV43} we have
\begin{align*}
    z_{k+1}\leq(1-r_{1,k})z_k+r_{2,k},\forall k\in\mathbb{N}.\addtag\label{eq:upperofV44}
\end{align*}
Since $h_k=h_0^k$ and $h_0\in(0,\frac{1}{t_1})$, we have $  h_0^k<\frac{1}{t_1^k}$. It is easy to know that $h_0^0<\frac{1}{t_1}$ when $k=0$. Suppose for any $k\in\mathbb{N}$, $h_0^k<\frac{1}{k+t_1}$ holds. Due to $t_1>\bar{c}_5>2$, we have
\begin{align*}
    h_0^{k+1}<\frac{1}{(k+t_1)t_1}<\frac{1}{2(k+t_1)}<\frac{1}{k+1+t_1},
\end{align*}
which means $h_0^k<\frac{1}{k+t_1},\forall k\in\mathbb{N}$ holds. From~$\omega_k=\beta_0(k+t_1)$ and $\eta_k=\frac{\beta_2}{\omega_k}$, one obtains that
\begin{align*}
   & r_{1,k}=\eta_km_{19}=\frac{a_1}{k+t_1}\addtag\label{eq:upperofV45}\\
   &r_{2,k}=m_{18}n\eta_k^2+m_{15}h_k^2\sigma_{\mathcal{C}}<\frac{a_2}{(k+t_1)^2}\addtag\label{eq:upperofV46}
\end{align*}
where $a_1=\beta_2 m_{19}/\beta_0$ and $a_2=n\beta_2^2m_{18}/\beta_0^2+m_{15}\sigma_{\mathcal{C}}$.
Since $\varphi_1<1$ and $\beta_1>1$, we have
\begin{align*}
   r_{1,k}\leq\frac{m_7}{2\beta_8}\leq\frac{\frac{\varphi_1}{2}+\frac{\varphi_1^2}{2}}{2(\frac{1}{2}+\beta_1)}\leq\frac{1}{3}\addtag\label{eq:upperofV47}
\end{align*}
From~\eqref{eq:upperofV44}--\eqref{eq:upperofV47} and~\eqref{eq:lemma61}, we have
\begin{align*}
    z_k \leq \phi_1\left(k, t_1, a_1, a_2, z_0\right), \forall k \in \mathbb{N}_+\addtag\label{eq:upperofV48}
\end{align*}
From~$\beta_0\geq \tilde{c}\nu\eta_2/4$ and $a_4=\mathcal{O}(n)$, we have
\begin{align*}
    \phi_1\left(k, t_1, a_1, a_2, z_0\right)= \begin{cases}\mathcal{O}\left(\frac{n}{k}\right), & \text { if } a_1>1, \\ \mathcal{O}\left(\frac{n \ln (k-1)}{k}\right), & \text { if } a_1=1, \\ \mathcal{O}\left(\frac{n}{k^{a_1}}\right), & \text { if } a_1<1,\end{cases}\addtag\label{eq:upperofV49}
\end{align*}
Since~\eqref{eq:proveoftheo21}, from~\eqref{eq:upperofV48} and~\eqref{eq:upperofV49}, one obtains that
\begin{align*}
    \|\x_k\|_\bk^2+V_{4,k}\leq n\bar{m}_3,\addtag\label{eq:upperofV50}
\end{align*}
for some constants $\bar{m}_3>0$.

Since~$\omega_k=\beta_0(k+t_1)$ and $\eta_k=\frac{\beta_2}{\omega_k}$, from~\eqref{eq:upperofV35},~\eqref{eq:lowerofw},~\eqref{upperofV2} and~\eqref{eq:upperofV50},  one obtains that
\begin{align*}
    \breve{z}_{k+1}\leq(1-a_3)\breve{z}_{k}+\frac{a_4}{(k+t_1)^2}\addtag\label{eq:upperofV51}
\end{align*}
where $\breve{z}_{k}=\E_{\xi_k}[\breve{V}_{k+1}]$, $a_3=\frac{1}{\beta_8}\min\{\bar{c}_4,\frac{m_5}{2},\frac{m_7}{2}\}$ and $a_4=n(m_{16}L_f^2\bar{m}_3+2L_f(2m_{16}+m_{17})\bar{m}_3+m_{18})\beta_2^2/\beta_0+m_{15}\sigma_{\mathcal{C}}$.

From~\eqref{eq:theo3w9}--\eqref{eq:theo3w12}, and~\eqref{eq:upperofV47}, it holds that
\begin{align*}
    0<a_3<1~\text{and}~a_4>0.\addtag\label{eq:upperofV52}
\end{align*}

From~\eqref{eq:upperofV51}, \eqref{eq:upperofV52} and~\eqref{eq:lemma63}, we have
\begin{align*}
    \breve{z}_k \leq \phi_2\left(k, t_1, a_3, a_4, z_0\right)=\mathcal{O}(\frac{n}{k^2})\addtag\label{eq:upperofV53}
\end{align*}

From~~\eqref{eq:lowerofw},~\eqref{upperofV2} and~\eqref{eq:upperofV53}, we have
\begin{align*}
    \E[\|\x_k\|_\bk^2]\leq\frac{1}{\check{c}_1}\phi_2\left(k, t_1, a_3, a_4, z_0\right)=\mathcal{O}(\frac{n}{k^2}).\addtag\label{eq:upperofV54}
\end{align*}
Then we have~\eqref{eq:convergenceofth31}.

From~\eqref{upperofV21} and~\eqref{eq:upperofV36}, we have
\begin{align*}
    \E[V_{4,k+1}]&\leq(1-\frac{\nu}{2}\eta_k)V_{4,k}+\frac{\eta_kL_f^2}{2}\|\x_k\|_\bk^2+\eta_k^2L_f\sigma^2\\
    &\leq (1-\frac{\nu\beta_2}{2\beta_0(k+t_1)})V_{4,k}+\mathcal{O}(\frac{n}{k^2})+\frac{\beta_2^2L_f\sigma^2}{\beta_0^2(k+t_1^2)}\addtag\label{eq:upperofV55}
\end{align*}
Thus, from~\eqref{eq:upperofV55}, the proof can be completed in the same way as the proof of~\cite[Lemma 5]{yi2022primal}.

\section{The proof of Theorem~\ref{theo:privacy1}}\label{app-privacy1}

Since the stochastic compressor $\mathcal{C}'(\cdot)$ satisfies Assumption~\ref{as:compressor}, it is straightforward to see from Appendices~\ref{app-convergence0} and~\ref{app-convergence1} that replacing the parameters $\varphi$ and $\sigma_{\mathcal{C}}$ with $\varphi(1-q)$ and $(1-q)\sigma_{\mathcal{C}}$, respectively, still guarantees the convergence of the RCP-SGD algorithm. To this end, we only need to prove the privacy under RCP-SGD. From Definition~\ref{def:differentialprivacy}, for any two adjacent sampled datasets $\mathcal{S}_i^{(1)}$ and $\mathcal{S}_i^{(2)}$, the $(0,\delta)$-differential privacy is achieved if
\begin{align*}
    \mathbb{P}\{\mathbb{M}(\mathcal{S}^{(1)},x_{-i})\in\mathcal{H}_i\}\leq \mathbb{P}\{\mathbb{M}(\mathcal{S}^{(2)},x_{-i})\in\mathcal{H}_i\}+\delta.
\end{align*}
We further denote $x_{-i,t}$ and $\mathcal{H}_{i,t}$ as the input and observation at time step $t$. Denote events $E_t^{(j)}=\{\mathbb{M}(\mathcal{S}_i^{(j)},x_{-i,t})\in\mathcal{H}_{i,t}\},~j=1,2,~\forall t\in\mathbb{N}$. Then the above inequality can be rewritten as
\begin{align*}
 \mathbb{P}\{\cup_{t=0}^\infty E_t^{(1)}\}\leq  \mathbb{P}\{\cup_{t=0}^\infty E_t^{(2)}\}+\delta.\addtag\label{eqn:privacy1}
\end{align*}
From the Definition~\ref{def:adjacency}, it can be observed that $\mathcal{S}^{(1)}_i$ and $\mathcal{S}^{(2)}_i$ only differ at a specific time step $k$, then we have
\begin{align*}
	\mathbb{P}\{\cup_{t=0,m\neq k}^\infty E_t^{(1)}\}=\mathbb{P}\{\cup_{t=0,t\neq k}^\infty E_t^{(2)}\}.\addtag\label{eqn:privacy2}
\end{align*}
By the law of conditional probability, Equation \eqref{eqn:privacy2} can be rewritten as
\begin{align*}
	p_1 \mathbb{P}\{\mathbb{M}(\mathcal{S}_i^{(1)},x_{-i,k})&\in\mathcal{H}_{i,k}\}\leq\\
	&p_1\mathbb{P}\{\mathbb{M}(\mathcal{S}_i^{(2)},x_{-i,k})\in\mathcal{H}_{i,k}\}+\delta.
\end{align*}
where $p_1=\mathbb{P}\{\cup_{m=0,m\neq k}^\infty E_m^{(1)}\}=\mathbb{P}\{\cup_{m=0,m\neq k}^\infty E_m^{(2)}\}$. From Definition~\ref{def:privacycom}, we have
\begin{align*}
	\sup_{\mathcal{H}_{i,k}}\mathbb{P}\{\mathbb{M}(\mathcal{S}_i^{(1)},x_{-i,k})\in\mathcal{H}_{i,k}\}-\mathbb{P}\{\mathbb{M}(&\mathcal{S}_i^{(2)},x_{-i,k})\in\mathcal{H}_{i,k}\}\\
	&\leq 1-\delta.
\end{align*}
 Consequently, it follows that $\delta\leq p_1(1-q)\leq 1-q$, which completes the proof.

\bibliographystyle{IEEEtran}
\bibliography{ref_Antai}

\end{document}